\documentclass[reqno,12pt]{amsart}
\usepackage{amsmath,amsthm,amssymb,amsfonts,amscd, textcomp}
\usepackage{epsfig}
\usepackage{color}
\usepackage[all]{xy}
\usepackage{tikz}
\usepackage{hyperref}
\theoremstyle{plain}
\newtheorem{theorem}{Theorem}

\newtheorem{lemma}[theorem]{Lemma}
\newtheorem{corollary}[theorem]{Corollary}
\newtheorem{proposition}[theorem]{Proposition}
\newtheorem*{theorem*}{Theorem}
\newtheorem*{conjecture*}{Conjecture}

\theoremstyle{definition}
\newtheorem{remark}[theorem]{Remark}
\newtheorem{example}[theorem]{Example}

\newtheorem{notation}[theorem]{Notation}

\usepackage{graphics, setspace}

\newcommand{\breakingcomma}{%
  \begingroup\lccode`~=`,
  \lowercase{\endgroup\expandafter\def\expandafter~\expandafter{~\penalty0 }}}
  
\usepackage{breqn}

\usepackage{amstext} 
\usepackage{array}
\newcolumntype{L}{>{$}l<{$}}

\allowdisplaybreaks

\newcommand{\CC}{{\mathbb{C}}}

\newcommand{\PP}{{\mathbb{P}}}
\newcommand{\QQ}{{\mathbb{Q}}}

\newcommand{\ZZ}{{\mathbb{Z}}}

\newcommand{\eps}{\varepsilon}

\newcommand{\SL}{{\mathrm{SL}}}

\newcommand{\Jac}{\mathrm{Jac}}
\newcommand{\Fix}{\mathrm{Fix}}

\newcommand{\id}{\mathrm{id}}

\newcommand{\bx}{{\bf x}}
\newcommand{\by}{{\bf y}}

\newcommand{\bz}{{\bf z}}

\newcommand{\hess}{\mathrm{hess}}
\newcommand{\age}{\mathrm{age}}

\def\A{{\mathcal A}}

\def\C{{\mathcal C}}

\def\Cl{{\mathrm{Cl}}}
\def\p{\partial }

\newcommand{\rmH}{{{\rm H}}}

\newcommand{\LL}{{\Upsilon}}
\DeclareMathOperator{\sgn}{sgn}

\newcommand{\ccHH}{{\mathsf{HH}}}

\newcommand\tet{\theta}

\newcommand\sig{\sigma}

\newcommand\ten{\otimes}

\begin{document}
\title[Hochschild cohomology of Fermat polynomials with non--abelian symmetries]{Hochschild cohomology of Fermat type polynomials with non--abelian symmetries}
\date{\today}
\author{Alexey Basalaev}
\address{A. Basalaev:\newline Faculty of Mathematics, National Research University Higher School of Economics, Usacheva str., 6, 119048 Moscow, Russian Federation, and \newline
Skolkovo Institute of Science and Technology, Nobelya str., 3, 121205 Moscow, Russian Federation}
\email{a.basalaev@skoltech.ru}
\author{Andrei Ionov}
\address{A. Ionov:\newline Faculty of Mathematics, National Research University Higher School of Economics, Usacheva str., 6, 119048 Moscow, Russian Federation, and \newline
Department of Mathematics, Massachusetts Institute of Technology, 77 Massachusetts
Ave., Cambridge, MA 02139, United States }
\email{aionov@mit.edu}

\begin{abstract}
For a polynomial $f = x_1^n + \dots + x_N^n$ let $G_f$ be the non--abelian maximal group of symmetries of $f$. This is a group generated by all $g \in \mathrm{GL}(N,\CC)$, rescaling and permuting the variables, so that $f(\bx) = f(g \cdot \bx)$.
For any $G \subseteq G_f$ we compute explicitly Hochschild cohomology of the category of $G$--equivarint matrix factorizations of $f$. We introduce the pairing on it showing that it is a Frobenius algebra.
\end{abstract}
\maketitle

\setcounter{tocdepth}{1}
\tableofcontents

\section{Introduction}

Equivariant approach to singularity theory had been under consideration in many different publications.
In particular, it appeared to be useful for such purposes as constructing new moduli spaces and integrable hierarchies in \cite{FJR}, cohomological field theories in \cite{PV} and quantum field theories in \cite{IV90}. One should notice that the certain equivariant approach ideas were applied in the classical singularity theory as well (cf. \cite{Or}). However, the role of equivariant approach to singularity theory is in particular important in mirror symmetry. 

%
Consider a pair $(f,G)$, where $f = f(\bx)$ is a polynomial having only isolated critical points and $G \subseteq \mathrm{GL}(N,\CC)$ a group of elements $g$, s.t. $f(g \cdot \bx) = f(\bx)$. In order to put mirror symmetry in a rigorous context, one has to provide the canonical construction of a Frobenius algebra of a pair $(f,G)$. 

\subsection*{Diagonal groups}
Let $G_f^d$ stand for {\em the group of maximal diagonal symmetries} of $f$.
$$
G_f^d := \left\{ (\lambda_1, \ldots , \lambda_N )\in(\CC^\ast)^N \, \left| \, f(\lambda_1x_1, \ldots, \lambda_n x_N) = 
f(x_1, \ldots, x_N) \right. \right\} .
$$ 
For $G \subseteq G_f^d$ the Frobenius algebra of $(f,G)$ was considered in \cite{K03, K09, BTW16, BTW17}. All these publications made an attempt to construct a new object employing some essential ideas --- it was widely agreed how the vector space of the pair $(f,G)$ should look and it remained to introduce the product structure and the pairing in a canonical manner. 
Such a product was given non-uniquely in a wide contex in \cite{K03} for all polynomials $f$ and $G \subseteq G_f^d$; in \cite{K09} it was fixed by the particular formula for the special class of $f$, so--called \textit{invertible polynomials} and $G \subseteq G_f^d \cap \SL(N)$; it was introduced axiomatically in \cite{BTW16} for all polynomials $f$ and $G \subseteq G_f^d$. It was also proved in \cite{BTW16} that a Frobenius algebra, satisfying such axioms is unique up to isomorphism.

However, there is an essential candidate for Frobenius algebra of $(f,G)$ --- Hochschild cohomology $\ccHH^\ast(\mathrm{MF}(f,G))$ of the category of $G$--equivarint matrix factorizations of $f$. The only reason why it was not widely used in mirror symmetry was computational --- this cohomology ring was very hard to compute.

In \cite{S20} D. Shklyarov developed the technique allowing one to compute $\ccHH^\ast(\mathrm{MF}(f,G))$ systematically. Despite the fact that this Hochschild cohomology is defined for the groups $G$ that are not necessarily abelian, Shklyarov's technique fully works only for $G \subseteq G_f^d$. 
This result of Shklyarov was used in \cite{BT2} for $f$ being invertible polynomial, to show that the construction of \cite{BTW16} provides the algebra isomorphic to $\ccHH^\ast(\mathrm{MF}(f,G))$ if $G \subseteq G_f^d$. 

\subsection*{Non-abelian groups}
Up to now very few results can be found in literature, concerning the pairs $(f,G)$ with a non--abelian symmetry group $G$. 
In \cite{EGZ18} the authors have considered the pairs $(f,G)$ topologically from the point of view of Milnor fibre. In \cite{PWW} the authors considered the A--model vector space for the certain examples of such pairs $(f,G)$. 

In particular, for $G$ not being diagonal, up to now there was no attempt to investigate a Frobenius algebra of a pair $(f,G)$. We address this question in the current paper. Namely, for fixed positive integers $N$ and $n$ we consider a polynomial
\[
  f=f_{N,n}(x_1,\ldots,x_N) := x_1^n+\ldots+x_N^n
\]
together with its group of symmetries $G_f = S_N \ltimes G_f^d$, where $S_N$ acts by permuting the coordinates $x_\bullet$. The pairs $(f_{N,n},G_f)$ were considered in the physical context in \cite{G91}. The certain approach towards the construction of a Frobenius manifold for such singularities was done in \cite{I16}.

\subsection*{In this paper}
We adopt the technique of Shklyarov to compute the basis and product of $\ccHH^\ast(\mathrm{MF}(f,G))$ for any $G \subseteq G_f$. According to this technique one has to consider the Hochschild cochain mixed complex of $(\CC[\bx]\rtimes G,f)$ with the coefficients in $\CC[\bx]\rtimes G$. Denoting its cohomology by $\ccHH^*(\CC[\bx] \rtimes G,f)$, we have $\ccHH^*(\CC[\bx] \rtimes G,f) \cong \ccHH^\ast(\mathrm{MF}(f,G))$. 
Consider also the Hochschild cochain mixed complex of $(\CC[\bx],f)$ with the coefficients in $\CC[\bx]\rtimes G$. Its cohomology $\ccHH^*(\CC[\bx],f; \CC[\bx] \rtimes G)$ has natural $G$--action, so that we have $\ccHH^*(\CC[\bx] \rtimes G,f) \cong \left( \ccHH^*(\CC[\bx],f; \CC[\bx] \rtimes G) \right)^G$. The $\cup$--product of $\ccHH^*(\CC[\bx],f; \CC[\bx] \rtimes G)$ can be computed combinatorially with the help of HKR isomorphism, found explicitly in \cite{S20}. However this isomorphism is not $G$--equivariant for non--abelian groups. This makes the problem of finding the $G$--action explicitly complicated. 

In this paper for any $G \subseteq G_f$ we resolve the $\cup$--product of $\ccHH^*(\CC[\bx],f; \CC[\bx] \rtimes G)$ via the explicit formulae. The latter cohomology ring can be endowed with the bigrading following the ideas of \cite{IV90}. We show that the $\cup$--product preserves this bigrading. Introducing the certain pairing we show also that $\ccHH^*(\CC[\bx],f; \CC[\bx] \rtimes G)$ is a Frobenius algebra.

Using the formulae for the $\cup$--product, we compute the $G$--action on ${\ccHH^*(\CC[\bx],f; \CC[\bx] \rtimes G)}$ explicitly. We show that $\ccHH^*(\CC[\bx] \rtimes G,f)$ is a Frobenius algebra too.

Even though we only deal with the polynomials of Fermat type, our approach is applicable to the case of arbitrary invertible polynomials as well. It's also an interesting question to consider the Hochschild cohomology of invertible polynomials in the context of Berglund-H\"ubsch-Henningson duality (cf. \cite{EGZ18,EGZ20}). We plan to do this is in a upcoming paper \cite{BI2}.

We also applied the results of the present paper to generalize the mirror map of \cite{K09} to the case of Fermat polynomials with non-abelian group of symmetries in \cite{BI21}.

The paper is organized as follows.
We fix notation and introduce some notions in Section~\ref{section: notations}. Section~\ref{section: vector space} is devoted to the definiton of the phase space $\A'_{f,G}$ --- a bigraded vector space with the pairing that will be later endowed with the $\cup$--product of Hochschild cohomology ring $\ccHH^*(\CC[\bx],f; \CC[\bx] \rtimes G)$. The structure constants of this product are introduced in Section~\ref{section: HH} where we explain the technique of Shklyarov. In Section~\ref{section: multiplication table} we develop a systematic approach how to compute these structure constants. It takes full Section~\ref{section: products in mixed sectors} to compute all structure constants needed. Finally in Section~\ref{section: G--action} we find the action of $G$ on $\A'_{f,G}$. This allows us to consider the $G$--invariants of $\A'_{f,G}$ and prove our main theorem showing that $\ccHH^*(\CC[\bx]\rtimes G,f)$ is a Frobenius algebra with the pairing introduced in Section~\ref{section: vector space}.
\subsection*{Acknowledgement}
~
\\
The work of A.B. was partially supported by International Laboratory of Cluster Geometry NRU HSE, RF Government grant, ag. N\textsuperscript{\underline{\scriptsize o}} 075-15-2021-608.
The work of A.I. was supported by RSF grant no. 19-71-00086.

The authors want to express a great respect to the anonymous referee for many corrections and deep comments.

\section{Preliminary notations}\label{section: notations}
Let $N$ and $n$ be positive integers. Consider a polynomial
\[
  f=f_{N,n}(x_1,\ldots,x_N) := x_1^n+\ldots+x_N^n.
\]
Such polynomials are called to be of Fermat type. The set of critical points of $f$ consists of $\bx = 0$ only and it makes sense to consider Jacobian algebra of $f$.
\[
    \Jac(f) := \CC[x_1,\dots,x_N] / (\p_{x_1}f,\dots,\p_{x_N}f).
\]
This is a finite--dimenional $\CC$--algebra with the canonical product that we denote by $\circ$. Moreover it is a \textit{Frobenius algebra}. Namely there is a non--degenerate symmetric bilinear form $\eta: \Jac(f) \otimes \Jac(f) \to \CC$, s.t. $\eta(a \circ b, c) = \eta(a ,b \circ c)$ for all $a,b,c \in \Jac(f)$. We will introduce this pairing explicitly beneath.
In what follows for any polynomial $p = p(\bx)$ we denote by $\lfloor p \rfloor$ its class in $\Jac(f)$.

The action of $S_N$ permuting coordinates on $\CC^N$ defines an action of $S_N$ on $G_f^d$ by permutation of diagonal entries. Let $G_f$ be the group of symmetries $G_f = S_N \ltimes G_f^d$ with the product $(\sigma',h)(\sigma,g) = (\sigma'\sigma,\sigma^{-1}(h)g)$. In what follows assume $g \in G_f^d$ being represented by a diagonal matrix $\mathrm{diag}(g_1,\dots,g_N)$. Let an element $(\sig,g)\in G_f$ act on $\bx \in \CC^N$ by $(\sig,g)(x_k) = g_kx_{\sig(k)}$. 

For any $G \subseteq G_f$ there are two natural subgroups of $G$
\begin{align*}
G^s :=\{ (\sig,\id)\} \quad \text{ and } \quad G^d := \{ (\id,g)\}. 
\end{align*}
In particular we have $(G_f)^s \simeq S_N$ and $(G_f)^d \simeq G_f^d$.
By abuse of notations we also write 
$\sig$ for an element $(\sig,\id)\in G^s$ and $g$ for an element $(\id,g)\in G^d$. Note that in this notation we have $(\sig,g)=\sig \cdot g$.

Denote by $(i_1,\ldots, i_k)$ the element $\sig\in S_N$ such that $\sig(i_a) = i_{a+1}$, $\sig(i_k)=i_{1}$ and $\sig(j)=j$ for $j\not\in\{i_1,\ldots, i_k\}$. 
We call an element $(\sig,g)\in G_f$ a {\itshape cycle of length $k$} if $\sig=(i_1,\ldots, i_k)$ and, moreover, $g_j=1$ for $j\not\in\{i_1,\ldots, i_k\}$. 

\begin{notation}
    In what follows denoting length $k$ cycles we assume $i_1$ to be the smallest element in $\{i_1,\ldots, i_k\}$.
\end{notation}


Let us fix a $n$-th primitive root of unity $\zeta_n :=\exp(2 \pi \sqrt{-1}/n)$. Denote by $t_i$ the element of  $G_f^d$ defined by 
\[
    t_i(x_j) = \begin{cases}
                \zeta_n \cdot x_i, \quad &\text{ if } i =j,
                \\
                1 \cdot x_j \quad &\text{ otherwise}.
               \end{cases}
\]
We call an element $(\sigma,g) \in G_f$ {\itshape special} if $g_1\ldots g_N=1$, i.e. if $g \in \mathrm{SL}_N(\CC)$, and \textit{non--special} otherwise.

\subsection{Fixed loci} \label{section: fixed sets}
For any $u \in G_f$ denote by $\Fix(u)\subseteq\CC^N$ the linear subspace of the fixed points of $u$, seen as an element of $\mathrm{GL}_N(\CC)$. For $u$ acting diagonally the set $\Fix(u)$ consists of the points with some of the coordinates set to zero. For $u = (\sigma,g)$ with $\sigma \neq \id$ the set $\Fix(u)$ is best seen after the diagonalization of $u$.

Denote $N_u := \dim_\CC\Fix(u)$ and $d_u\colon=N-N_u$.
For each $u \in G_f$ denote by $I_{u}^c$ the set of all indices $i$, s.t. $u(x_i) \neq x_i$, and by $I_{u}$ the complement of $I_{u}^c$ in $\{1,\dots,N\}$. 
In particular, for $g\in G_f^d$ we get $I_g = \{i_1,\dots,i_{N_g}\}$ --- the set of all indices $j$ such that $g_j=1$ and $I_{\id}=\{1,\dots, N\}$.

The elements $u_1,\dots,u_k \in G_f$ will be called \textit{non--intersecting} if $I_{u_i}^c \cap I_{u_j}^c = \emptyset$ for all~$i \neq j$.

\begin{remark}
    Trying to keep the notation consistent with all the other work done before us we have to note that the notation for $I_g$ in \cite{BTW16,BTW17,BT2} is opposite to \cite{S20}. We stick to that being historically first.
\end{remark}

Decompose further $\Fix(u) = \Fix^T(u) \oplus \Fix^E(u)$ into a \textit{tautological fixed locus} $\Fix^T(u) := \{x\in\CC^N~|~x_{j}=0, j\in I_{(\sig,g)}^c \}$ and a \textit{eigen fixed locus} $\Fix^E(u)$.
For any element $u \in G_f$ denote by $f^{u}$ the restriction of $f$ to the fixed locus $\Fix(u)$. 

The following example will describe the fixed locus of a special cycle. This example will be used throughout the paper.

\begin{example}\label{example: special cycle}
Consider special cycle $(\sigma,g) \in G_f$ with ${\sigma = (i_1,\dots,i_k)}$. Recall that in our convention $g$ acts trivially on $x_j$ with $j \not\in \{i_1,\dots,i_k\}$. 

For $\zeta_k = \exp(2 \pi \sqrt{-1} /k)$ consider the change of the variables
\[
  \widetilde x_{i_b} := \frac{1}{k}\sum_{a=1}^k \zeta_k^{(1-b)(a-1)} \widetilde g_{i_a} x_{i_a}, \quad 1 \le b \le k,
\]
with $\widetilde g_{i_1} := 1$ and $\widetilde g_{i_a} := g_{i_1}\cdot\dots\cdot g_{i_{a-1}}$ for $ 1 \le a \le k$.
Then we have
\[
  (\sigma,g)(\widetilde x_{i_b}) = \zeta_k^{b-1} \widetilde x_{i_b}.
\]
In particular, the variable $\widetilde x_{i_1}$ is preserved by the action of $(\sigma,g)$. We have:
\[
    x_{i_a} = \frac{1}{\widetilde g_{i_a}} \sum_{b=1}^k \zeta_k^{(b-1)(a-1)} \widetilde x_{i_b}, \quad 1 \le a \le k,
\]
and also
\[
  f^{(\sigma,g)} = k \cdot \widetilde x_{i_1}^n + \sum_{\substack{l \not\in \{i_1,\dots,i_k\} \\ 1 \le l \le N}} x_l^n.
\]
We have 
\begin{align}
    \Fix^T((\sigma,g)) &= \{x \in \CC^N | \ x_{i_1} = \dots = x_{i_k} = 0 \},
    \\
    \Fix^E((\sigma,g)) &= \{ x \in \CC^N | \ x_j = 0, \ j \not\in \{i_1,\dots,i_k \} \text{ and } x_{i_a} = t \cdot \widetilde g_{i_a}^{-1}, \  t\in \CC \}. 
\end{align}
For every $x_{i_a}$, considered as a function of $\widetilde x_\bullet$ we have  
\begin{align}
&\left. x_{i_1}\right|_{ \Fix((\sigma,g))} =\left. \widetilde g_{i_2} x_{i_2}\right|_{ \Fix((\sigma,g))}=\ldots=\left. \widetilde g_{i_k} x_{i_k}\right|_{ \Fix((\sigma,g))}=\left. \widetilde x_{i_1}\right|_{ \Fix((\sigma,g))},
\label{formula: coordinates on sector}
\\
& \left. \widetilde x_{i_b}\right|_{ \Fix((\sigma,g))}=0 \enskip \text{for $b>1$.}
\label{formula: coordinate vanishing on sector} 
\end{align} 
From this we conclude $\lfloor \left( \widetilde g_{i_p} x_{i_p} \right) ^q \rfloor =  \lfloor \widetilde x_{i_1}^q \rfloor$ in $\Jac(f^{(\sigma,g)})$.

In what follows denote 
\[
    \widetilde x_{(\sigma,g)} := \widetilde x_{i_1} = \frac{1}{k}\sum_{a=1}^k \widetilde g_{i_a} x_{i_a}.
\]
\end{example}

The following lemma is obvious.
\begin{lemma}\label{lemma: fixed set}
We have
\begin{description}
 \item[(1)] If $u \in G_f$ is a non--special cycle or $u \in G^d$. Then $\Fix^E(u) = 0$.
 \item[(2)] If $u \in G_f$ satisfies $u = \prod_{i=1}^k (\sigma_i,g_i)$ for some nonintersecting cycles $(\sigma_i,g_i)$. Then 
 \[
    \Fix^T(u) = \cap_{i=1}^k \Fix^T((\sig_i,g_i)) \quad \text{ and } \quad \Fix^E(u) = \bigoplus_{i=1}^k \Fix^E((\sig_i,g_i)).
 \] 
\end{description}
\end{lemma}


It follows from the lemma above that for any $u \in G_f$, $f^u$ either vanishes or is a polynomial of Ferma type again. In particular, it defines an isolated singularity in the latter case and it makes sense to consider its Jacobian algebra $\Jac(f^u)$.

For $f^u \equiv 0$ we set $\Jac(f^u) := \CC \cdot [1]$ --- the $1$--dimensional $\CC$--algebra generated by~$[1]$.


\section{The phase space}\label{section: vector space}
Let $G \subseteq G_f$. For any $u \in G$ let $\xi_u$ be a formal letter and $f^u$ stand for the restriction of $f$ to $\Fix(u)$ as in section above. 
Consider the $\CC$--vector space
\begin{equation}
    \A'_{f,G} := \bigoplus_{u \in G} \A'_{u} \text{ with } \A'_{u} := \Jac(f^u)\xi_u.
\end{equation} 
We will denote by $\lfloor \phi \rfloor \xi_u$ the elements of $\A'_{u}$, where $\lfloor \phi \rfloor \in \Jac(f^u)$. 
The subspace $\A'_{u}$ will be called \textit{$u$--sector}.


The vector space $\A'_{f,G}$ is $G$--graded by its construction. Introduce also the $\ZZ/2\ZZ$ grading on it by defining $\lfloor \phi \rfloor\xi_u$ to be of parity $\overline 0$ if $N - N_u \equiv 0 \mod 2$ and of parity $\overline 1$ if $N - N_u \equiv 1 \mod 2$.

Following Shklyarov~\cite{S20} we are going to equip this algebra with the product 
\begin{align}
    \cup: \A'_{f,G} \otimes \A'_{f,G} &\to \A'_{f,G}
    \\
    \lfloor \phi(\bx) \rfloor\xi_u \cup \lfloor \psi(\bx) \rfloor\xi_{v} &= \lfloor \phi(\bx)\psi(u(\bx)) \sigma_{u,v}\rfloor \xi_{uv},\quad \phi(\bx), \psi(\bx)\in \CC[\bx],
    \label{eq: HH cup}
\end{align}
where $\sigma_{u,v}$ are some polynomials that will be introduced in Section~\ref{section: HH}. 

It's important to note that the product above assumes the classes of polynomials in the different quotient rings. 
Special polynomials $\sigma_{u,v}$ are needed in order to make this product well-defined. These are given by Hochschild cohomology ring.
In particular, it follows that one may apply Eq.~\eqref{formula: coordinates on sector} and \eqref{formula: coordinate vanishing on sector} on both sides of Eq.\eqref{eq: HH cup} without any affect to the classes.
We illustrate this property in Remark~\ref{remark: product well-defined} after computing all $\sigma_{u,v}$.

One notes immediately that $\cup$--product respects the $G$--grading. It also respects the $\ZZ/2\ZZ$--grading. Namely, we have $\sigma_{u,v} = 0$ if $N-N_u + N-N_v \not\equiv N- N_{uv} \mod 2$.

\subsection{Pairing}
For $(\id, g) \in G$ we define the polynomial 
\[
 \mathcal{H}_{(\id, g)} :=  \prod n(n-1) x_a^{n-2},
\]
the product being taken over all $a \in I_{(\id,g)}^c$.

For a cycle $(\sig, g)$ with $\sig=(i_1, \ldots, i_k)$ we define the polynomial 
\[
 \mathcal{H}_{(\sig, g)} := 
 \begin{cases}
    (-1)^{k-1}n(n-1)\widetilde{x}_{(\sig, g)}^{n-2}, &\quad \text{ if $g$ is special}, 
    \\
    \det(g)^{-1}-1 &\quad \text{ otherwise,}
 \end{cases}
\]
for $\widetilde x_{(\sig, g)}$ as in Example~\ref{example: special cycle}.

For general $u\in G$ let $u = (\id,g_0)\cdot\prod_{i=1}^k (\sigma_i,g_i)$ be its decomposition into non--intersecting cycles. 
Denote
\[
p_i := \begin{cases}
        |\sigma_i|-1 \quad &\text{if $g_i$ is special},
        \\
        |\sigma_i| \quad &\text{if $g_i$ is non-special}.
        \end{cases}
\]
Then we set
\[
    \mathcal{H}_u := (-1)^{D} \prod_{i=0}^k \mathcal{H}_{(\sigma_i,g_i)}, \quad D = \sum_{i=1}^{k-1} p_i (p_{i+1} + \dots + p_k)=\sum_{1\le i<j\le k} p_ip_j.
\]

Up to a multiplicative constant we can see $\mathcal{H}_u$ as a Hessian determinant of $f^u$.

For an element $u\in G$ there is a non-degenerate symmetric bilinear form $\eta_{f^u}$ on $\Jac(f^u)$ defined up to multiplicative constant and endowing it with the structure of Frobenius algebras. We fix this constant by setting 
\[
\eta_{f^u}(\lfloor 1\rfloor,\lfloor \mathcal{H}_u\rfloor) = (n-1)^{N_u}.
\]
Moreover, $\eta_{f^u}(\lfloor 1\rfloor,-)$ vanishes on monomials in $\Jac(f^u)$ not proportional to $\lfloor \mathcal{H}_u\rfloor$. These properties define $\eta_{f^u}$ uniquely.

Consider the direct sum decomposition of $\A'_{f,G}$. We introduce the $\CC$--bilinear form $\eta_{f,G}$ on it by
\[
    \eta_{f,G} := \bigoplus_{u \in G} \eta_u,
\]
where \[ \eta_u \left(\lfloor \phi' \rfloor \xi_u, \lfloor \phi'' \rfloor \xi_{u^{-1}} \right)=\eta_{f^u}(\lfloor \phi' \rfloor, \lfloor \phi'' \rfloor). \]

It follows immediately from the definition that $\eta_{f,G}$ is non--degenerate. Moreover, it satisfies the following symmetry condition: if $u=(\sig, g)$ then
\[
\eta_{u}\left(\lfloor \phi' \rfloor \xi_u, \lfloor \phi'' \rfloor \xi_{u^{-1}} \right)=(-1)^{M_u}\det(g)\eta_{u^{-1}}\left(\lfloor \phi'' \rfloor \xi_{u^{-1}}, \lfloor \phi' \rfloor \xi_{u} \right), 
\]
where $M_u$ is the number of nonspecial cycles in $u$. This condition holds as we have \[(-1)^{M_u}\det(g)\mathcal{H}_u=\mathcal{H}_{u^{-1}}\] by construction of $\mathcal{H}_u$. In particular $\eta_{f,G}$ is $\ZZ/2\ZZ$--symmetric with $\ZZ/2\ZZ$-grading of $\lfloor\phi\rfloor\xi_u$ given by parity of $M_u$  if and only if $G^d\subseteq\SL_N(\CC)$.
We show in Theorem~\ref{theorem: Frobenius algebra before invariants} that if $G^d\subseteq\SL_N(\CC)$ it endows $\A'_{f,G}$ with a structure of Frobenius algebra and in Proposition~\ref{proposition: form G-invariant} that it is $G$-invariant.

\subsection{Bigrading}\label{section: bigrading}
The function $f$ assumed in this paper is quasi-homogeneous. This provides the grading of the polynomial ring $\CC[\bx]$. In what follows denote it by $\deg$.

For any $u \in G$ let $\lambda_1,\dots,\lambda_N \in \CC$ be the eigenvalues of the linear transformation $\CC^N \to \CC^N$ given by $\bx \mapsto u \bx$. We may assume $\lambda_k = \exp(2 \pi \sqrt{-1} a_k)$ for some $a_k \in \QQ \cap [0,1)$. Denote:
\[
    \age(u) := \sum_{k=1}^N a_k.
\]
Then for the inverse element $u^{-1}$ we have
\[
    \age(u) + \age(u^{-1}) = N - N_{u} = d_u.
\]

For any homogeneous polynomial $\phi = \phi(\bx)$ assume the element $\lfloor \phi \rfloor \xi_{u} \in\A'_{f,G}$. Define its {\itshape left charge} $q_l$  and {\itshape right charge} $q_r$ to be 
\[
    \left(q_l(\lfloor \phi\rfloor \xi_{u}), q_r(\lfloor \phi\rfloor \xi_{u}) \right) = \left( \frac{\deg \phi- d_{u}}{n}+{\rm\age}{(u)}, \frac{\deg \phi- d_{u}}{n}+{\rm\age}{(u^{-1})} \right).
\]
This definition endows $\A'_{f,G}$ with the structure of a $\QQ$-bigraded vector space. 
This is exactly the bigrading introduced in \cite{Muk, IV90}. 

It follows immediately that $q_\bullet (\xi_u) + q_\bullet (\xi_v) = q_\bullet (\xi_{uv})$ for $u,v \in G$, s.t. $\sigma_{u,v} \in \CC$. In particular, this holds for non--intersecting $u,v$ as will be explained later.
We show in Proposition~\ref{proposition: cup product preserves bigrading} that $\cup$--product preserves the bigrading introduced. 

\begin{example}
Consider a special cycle $(\sigma,g) \in G$ with $\sigma = (i_1,\dots,i_k)$. It acts diagonally in the basis dual to $\tilde x_{i_a}$ (see Example~\ref{example: special cycle}). We get the eigenvalues $\{ \lambda_{i_1}, \dots, \lambda_{i_k}\} = \{1,\zeta_k,,\zeta_k^2,\dots,\zeta_k^{k-1} \}$ and $\lambda_b = 1$ for all $b \not \in \{i_2,\dots,i_k\}$. We get $\age(\sigma,g) = k(k-1)/ (2k)$ and $\age(\sigma,g)^{-1} = k-1 - k(k-1)/ (2k) = (k-1)/ 2$.
Therefore we have 
\[
    \left(q_l(\xi_{(\sigma,g)}), q_r(\xi_{(\sigma,g)}) \right) = \left(\frac{k-1}{2} - \frac{k-1}{n}, \frac{k-1}{2} - \frac{k-1}{n} \right).
\]
\end{example}

\begin{example}
Consider a non--special $(\sigma,g) \in G$ with $\sigma = (i_1,\dots,i_k)$. Let 
$g \cdot x_p = \zeta_n^{a_p} x_p$
with $a_p = 0$ for $p \not\in \{i_1,\dots,i_k\}$. Assume also $\sum_p a_p < n$. The eigenvalues are $\lambda_{i_p} = \zeta_k^{p} \cdot \exp \left(2 \pi \sqrt{-1} \sum_{i=1}^k a_i / k\right)$ , $p = 1,\dots,k$ and $\lambda_b = 1$ for all $b \not \in \{i_1,\dots,i_k\}$. 
We get 
\begin{align}
\age(\sigma,g) & = \sum_i \frac{a_i}{N}  + \frac{k-1}{2} = \age(g) + \frac{k-1}{2}
\\
\age(\sigma,g)^{-1} &= k - \age(g) - \frac{k-1}{2} = \frac{k+1}{2} - \age(g).
\end{align}
Therefore we have 
\[
    \left(q_l(\xi_{(\sigma,g)}), q_r(\xi_{(\sigma,g)}) \right) = \left(\frac{k-1}{2} +\age(g), \frac{k+1}{2} - \age(g) \right).
\]
\end{example}

\section{Hochschild cohomology via Koszul complex}\label{section: HH}
In this section we present the technique of Shklyarov \cite{S20}. In particulr, we follow the notation and exposition of loc.cit.

Define the $N$-th $\ZZ$-graded {\em Clifford algebra} $\Cl_N$ as the quotient algebra of 
\[
\CC\langle\theta_1,\ldots,\theta_N,{\partial_{\theta_1}},\ldots,{\partial_{\theta_N}}\rangle
\]
modulo the ideal generated by 
\[
\theta_i\theta_j=-\theta_j\theta_i,\quad{\partial_{\theta_i}}{\partial_{\theta_j}}=-{\partial_{\theta_j}}{\partial_{\theta_i}},\quad{\partial_{\theta_i}}\theta_j=-\theta_j{\partial_{\theta_i}}+\delta_{ij},
\]
where $\theta_i$ is of degree $-1$ and ${\partial_{\theta_i}}$ is of degree $1$.
For $I\subseteq\{1,\ldots,N\}$ write
\begin{equation}
\partial_{\theta_{I}}:=\prod_{i\in I}\partial_{\theta_i}, 
\quad {\theta_{I}}:=\prod_{i\in I}{\theta_i},
\end{equation}
where in both cases  the multipliers are taken in increasing order of the indices. 
The subspaces $\CC[{\theta}]=\CC[\theta_1,\ldots,\theta_N]$ and $\CC[\partial_{\theta}]=\CC[\partial_{\theta_1},\ldots,\partial_{\theta_N}]$ of $\Cl_N$ have the left $\ZZ$-graded $\Cl_N$-module structures via the isomorphisms  
\[
\CC[{\theta}]\cong \Cl_N/\Cl_N\langle{\partial_{\theta_1}},\ldots,{\partial_{\theta_N}}\rangle, \quad \CC[\partial_{\theta}]\cong \Cl_n/\Cl_n\langle{{\theta_1}},\ldots,{{\theta_N}}\rangle. 
\]  

Write $\CC[x_1,\dots,x_N,y_1,\ldots, y_N]$ as $\CC[\bx,\by]$ and 
$\CC[x_1,\dots,x_N,y_1,\ldots, y_N, z_1,\dots, z_N]$ as $\CC[\bx,\by,\bz]$. 
For each $1 \le i \le N$, there is a map
\begin{eqnarray}\label{deltai}
\nabla^{\bx\to (\bx,\by)}_i: \CC[\bx]\to \CC[\bx,\by],\qquad \nabla_i(p):=\frac{l_i(p)-l_{i+1}(p)}{x_i-y_i}.
\end{eqnarray}
where $l_i(p):=p(y_1,\ldots,y_{i-1},x_i,\ldots,x_N)$, $l_1(p) = p(\bx)$ and $l_{N+1}(p) = p(\by)$. 
They are called the \textit{difference derivatives}, whose key property is the following: 
\begin{equation}\label{diffder}
\sum_{i=1}^N(x_i-y_i)\nabla_i(p)=p(\bx)-p(\by).
\end{equation}

The difference derivatives can be applied consecutively. In particular, we shall use 
$\nabla_i^{\by\to(\by,\bz)}\nabla_j^{\bx\to(\bx,\by)}(p)$, which is an element of $\CC[\bx,\by,\bz]$.  
For any $\CC$-algebra homomorphism $\psi: \CC[\bx] \to \CC[\bx]$, write $\nabla_i^{\bx\to(\bx,\psi(\bx))}(p) := \left.\nabla_i^{\bx\to(\bx,\by)}(p) \right|_{\by = \psi(\bx)} \in \CC[\bx]$.

\subsection{Clifford basis of $\A'_{f,G}$}
In order to make use of Shklyarov's technique we need to represent the basis elements of $\A'_{f,G}$ via some elements of $\Cl_N \otimes \CC[\bx]$ as above. 
To do this for every $u \in G$ we introduce the ordered set $J_u$ that will consist of the same entries as $I_u^c$ does, but be ordered differently. 

First consider some special cases.

\subsubsection{Length $k$ special cycle}
Consider special cycle $(\sigma,g) \in G_f$ with $\sigma = (i_1,\dots,i_k)$. Let also $\widetilde x_{i_a}$ be as in Section~\ref{section: fixed sets}.
Corresponding to the linearization assumed there, we also have
\begin{align}
  \widetilde \theta_{i_b} &= \frac{1}{k}\sum_{a=1}^k \zeta^{(1-b)(a-1)} \widetilde g_{i_a} \theta_{i_a}, 
  \quad 
  \p_{\widetilde \theta_{i_b}} = \sum_{a=1}^k \zeta^{(b-1)(a-1)} \frac{1}{\widetilde g_{i_a}} \p_{\theta_{i_a}}, \quad 1 \le b \le k,
  \\
  & \p_{\theta_{i_a}} = \frac{\widetilde g_{i_a}}{k} \sum_{b=1}^k \zeta^{(1-b)(a-1)} \p_{\widetilde \theta_{i_b}}, \quad 1 \le a \le k.
\end{align}

Set $J_{(\sigma,g)} := \{i_1,\dots,i_k\}$ to be the ordered set of indices (recall that according to our notation this means also that $i_1 < i_a$ for all $a \ge 2$) and 
\begin{align*}
  \widetilde\theta_{(\sigma,g)} & := \widetilde\theta_{J_{(\sigma,g)}} := \sum_{m=1}^k (-1)^{m-1} \widetilde g_{i_m} \theta_{i_1} \cdot \dots \widehat{m} \dots \cdot \theta_{i_k},
  \\
  \p_{\widetilde\theta_{(\sigma,g)}} & := \p_{\widetilde\theta_{J_{(\sigma,g)}}} := \sum_{m=1}^k (-1)^{m-1} \widetilde g_{i_m} \p_{\theta_{i_1}} \cdot \dots \widehat{m} \dots \cdot \p_{\theta_{i_k}},
\end{align*}
where $\widehat{m}$ stands for the skipped index $m$ multiple.

Note that with $J_{(\sigma,g)}$, ordered by the actions of $\sigma$ and the set ${\overline{J}_{(\sigma,g)} := \{j_1,\dots,j_k\}}$ being the same set ordered ascending we have
\begin{align}
  \p_{\tilde\theta_{(\sigma,g)}} &= \sum_{m=1}^k (-1)^{m-1} \widetilde g_{i_m} \p_{\theta_{i_1}} \cdot \dots \widehat{m} \dots \cdot \p_{\theta_{i_k}} = \left( \sum_{m=1}^k \widetilde g_{i_m}\theta_{i_m} \right)\cdot \prod_{m=1}^k \p_{i_m}
  \\
  & = \sgn(\sigma) \left( \sum_{m=1}^k \widetilde g_{i_m}\theta_{i_m} \right)\cdot \prod_{m=1}^k \p_{j_m} = \sgn(\sigma) \sum_{m=1}^k (-1)^{m-1} \widetilde g_{j_m} \p_{j_1} \dots \hat m \dots \p_{j_k}.
\end{align}

\subsubsection{Length $k$ non--special cycle} 
Assume non--special $(\sigma,g) \in G$ with $\sigma = (i_1,\dots,i_k)$. Set $J_{(\sigma,g)} := \{i_1,\dots,i_k\}$ and
\begin{align*}
  \widetilde\theta_{(\sigma,g)} & := \widetilde\theta_{J_{(\sigma,g)}} := \theta_{i_1} \cdot \dots \cdot \theta_{i_k},
  \\
  \p_{\widetilde\theta_{(\sigma,g)}} & := \p_{\widetilde\theta_{J_{(\sigma,g)}}} := \p_{\theta_{i_1}} \cdot  \dots \cdot \p_{\theta_{i_k}}.
\end{align*}

\subsubsection{Diagonal element} For $(\id,g) \in G$ define $J_{(\id,g)} := I_g^c$. For $I_g^c = \{i_1,\dots,i_k\}$ set
\begin{align*}
  \widetilde\theta_{(\id,g)} & := \widetilde\theta_{J_{(\id,g)}} := \theta_{i_1} \cdot \dots \cdot \theta_{i_k},
  \\
  \p_{\widetilde\theta_{(\id,g)}} & := \p_{\widetilde\theta_{J_{(\id,g)}}} := \p_{\theta_{i_1}} \cdot  \dots \cdot \p_{\theta_{i_k}},
\end{align*}

\subsubsection{Generic element}
Finally for an arbitrary $u \in G$ consider its decomposition into non--intersecting elements $u = \prod_{p=1}^r u_p$ with $u_p = (\sigma_p,g_p) \in G_f$ being either diagonal, non--special or special cycles. The elements $u_a$ and $u_b$ commute for different $a,b$ and we may assume that $\min J_{u_a} < \min J_{u_b}$ whenever $a < b$.
Denote $J_p := J_{u_p}$ and set
\[
    \p_{\widetilde\theta_u} :=  \p_{\widetilde\theta_{J_1}} \cdot \dots \cdot \p_{\widetilde\theta_{J_r}}.
\]

For every $u_p$ consider the linearization as above and also the corresponding eigenvalues $\lambda_\bullet$.
Let $\rmH_{{f,u}}(\bx)$ be the element of $\CC[\bx]\otimes \CC[\theta]$ given by
\begin{eqnarray}\label{del-3}
\rmH_{{f,u}}(\bx):=\sum_{{i,j\in I_{u}^c,\,\,  j<i}}\frac{1}{1-\lambda_j}\nabla^{\bx\to(\bx,\bx^u)}_j\nabla^{\bx\to(\bx,u(\bx))}_i(f)\,\widetilde\theta_j\,\widetilde\theta_i,
\end{eqnarray}
where $\bx^u$ is restriction to $\Fix(u)$.

We attribute to $\xi_u \in \A'_u$ the following Clifford element
\[
    \widetilde\xi_u := \exp\left(\rmH_{f,u}\right) \cdot \p_{\widetilde\theta_u} = \p_{\widetilde\theta_{J_1}} \cdot \dots \cdot \p_{\widetilde\theta_{J_r}} + \dots.
\]

In what follows we will call the degree of $\p_{\widetilde\theta_u}$ the \textit{length} of element $\xi_u$.
One notes immediately that the parity of $\xi_u$ defined in Section~\ref{section: vector space} is equal to length of $\xi_u$ modulo $2$.

\subsection{Product}
For each pair $(u,v)$ of elements in $G$, define the class $\lfloor \sigma_{u,v} \rfloor \in \Jac(f^{uv})$ as follows.

Let $\rmH_f(\bx,u(\bx),\bx)$ be the element of $\CC[\bx]\otimes\CC[\theta]^{\otimes2}$ defined as the restriction to the set 
$\{\by=u(\bx),\, \bz=\bx\}$ of the following element of  $\CC[\bx,\by,\bz]\otimes\CC[\theta]^{\otimes 2}$ 
\begin{equation}\label{del-2}
\rmH_{f}(\bx,\by,\bz):=\sum_{1\leq j\leq i\leq n} \nabla^{\by\to(\by,\bz)}_j\nabla^{\bx\to(\bx,\by)}_i(f)\,\theta_i\otimes \theta_j.
\end{equation}
We can consider the series expansion of $\exp(\rmH_{f})$ above with respect to the natural product on  $\CC[\bx]\otimes \CC[\theta]\otimes \CC[\theta]$. 
It follows immediately that for our particular choice of the polynomial $f$ we have
\begin{align}
  \rmH_f(\bx,\by,\bz) &= \sum_{i=1}^N \sum_{a+b+c = n-2} x_i^a y_i^b z_i^c \ \theta_i \otimes \theta_i,
  \\
  \rmH_f(\bx,(\sig,g)(\bx),\bx) &= \sum_{i=1}^N\sum_{a+b=n-2}(a+1)g_i^bx_i^ax_{\sig(i)}^b\tet_i\ten\tet_i.
\end{align}

Let 
$\LL$ be the $\CC[\bx]$-linear extension of the degree zero map ${\CC[{\theta}]^{\otimes2}\otimes \CC[\partial_{\theta}]^{\otimes2}\to\CC[\partial_{\theta}]}$ defined by
\begin{equation}\label{mu}
 p_1(\theta)\otimes p_2(\theta)\otimes q_1(\partial_\theta)\otimes q_2(\partial_\theta)\mapsto(-1)^{|q_1||p_2|}p_1(q_1)\cdot p_2(q_2)
\end{equation}
where $p_i(q_i)$ denotes the action of $p_i(\theta)$ on $q_i(\partial_\theta)$ via the $\Cl_N$-module structure on $\CC[\partial_\theta]$ defined above and $\cdot$ is the natural product in $\CC[\partial_{\theta}]$. 

Define $\sigma_{u,v}$ to be the coefficient of $\widetilde\xi_{uv}$ in the following expression
\begin{equation}\label{eq: sigma g,h}
    \LL\Big( \ \exp \left(\rmH_f(\bx,u(\bx),\bx) \right) \cdot \widetilde\xi_u \otimes \widetilde\xi_v \Big).
\end{equation}

It follows immediately from the definition that $\sigma_{u,v}$ is zero unless ${d_{u,v}:=\frac{1}{2}(d_u +d_v - d_{uv})}$ is a non--negative integer.

\begin{theorem}[\cite{S20}]\label{theorem: Shklyarov}
    For any $G \subseteq G_f$ there is a $\CC$-algebra isomorphism 
    \[
        \ccHH^*(\CC[\bx],f; \CC[\bx]\rtimes G) \cong \A'_{f,G}
    \] 
    preserving both $\ZZ/2\ZZ$--grading and $G$--grading. 
    
    This isomorphism gives
    \[
        \ccHH^*(\CC[\bx]\rtimes G,f) \cong \left( \A'_{f,G} \right)^G.
    \] 
\end{theorem}

One should note that Shklyarov only makes an explicit statement for the diagonal groups $G$. However, this result holds for the symmetry groups we consider as well (cf. Remark~4.14 and Remark~4.16 in \cite{S20}).

\subsection{$G$ action on $\A'_{f,G}$}\label{section: G on HH from def}
The space $\ccHH^*(\CC[\bx],f; \CC[\bx]\rtimes G)$ is naturally endowed with the action of $G$. By the theorem above this $G$--action is also defined for $\A'_{f,G}$. Namely, for any $u,v \in G$ there is a $\CC$--linear map $v^*: \A'_{f,u} \to \A'_{f,vuv^{-1}}$, s.t. the following \textit{braided} commutativity property holds
\begin{align}
    \lfloor\phi(\bx)\rfloor\xi_u \cup \lfloor\psi(\bx)\rfloor\xi_v &= (-1)^{|\xi_u||\xi_v|} \lfloor\psi(\bx)\rfloor\xi_v \cup (v^{-1})^*\left( \lfloor\phi(\bx)\rfloor\xi_u \right)
    \\
    &= (-1)^{|\xi_u||\xi_v|} \lfloor\psi(\bx)\rfloor\xi_v \cup \lfloor\phi(v^{-1}(\bx))\rfloor (v^{-1})^*\left(\xi_u \right).
    \label{eq: braided commutativity}
\end{align}

%
Assume $u,v \in G$ are simultaneously diagonalized. Let $\lambda_i^u$ and $\lambda_i^v$ be the eigenvalues of $u,v$ in the basis of eigenvectors of $u$.
We have
\begin{equation}\label{eq: G--action on diagonalizable}
    v: \quad \lfloor\phi(\bx))\rfloor \xi_u \mapsto \prod_i (\lambda_i^v)^{-1} \,\cdot  \lfloor\phi(v(\bx))\rfloor \xi_{vuv^{-1}},
\end{equation}
where the product is taken over $i$, s.t. $\lambda_i^{u} \neq 1$.
In particular, we have 
\begin{equation}\label{eq: G--action on id}
    g^* \left( \lfloor\phi(\bx)\rfloor \xi_\id \right) = \lfloor\phi(g(\bx))\rfloor \xi_\id
\end{equation}

In what follows in order to compute this $G$--action we use the braided--commutativity property above. Namely, we compute all the products of both sided of this equality using Eq.~\eqref{eq: HH cup} and show that these products fix the $G$--action uniquely.


\subsection{Brief review of Shklyarov's technique}
In order to prove Theorem~\ref{theorem: Shklyarov} Shklyarov makes the following steps.

On the first step (Section 4.1.3) Shklyarov establishes an explicit quasi-isomorphism between the Hochschild complex and mixed Koszul complex. Both complexes are assumed with the coefficients in any ring $M$, in particular, the ring $M = \CC[\bx]\rtimes G$ suits too. Important achievement here is also the formula, connecting the products in the cohomology rings of the two complexes. 

On the second step (Section 4.1.6) Shklyarov relates Koszul complex to another mixed complex. Namely, he considers the ``Clifford'' complex $M[\p_{\theta_1},\dots,\p_{\theta_N}]$ with the certain differentials, isomorphic to the Koszul complex. Here $M$ is still arbitrary, and the product in cohomology ring of Hochschild complex is related to the product on the Clifford complex via the certain formula. This is exactly the place from where the operators $\exp(\rmH_{f})$ and $\LL$ come. 

On the third step (Section 4.2) Shklyarov considers the certain coefficient ring $M = \CC[\bx]\rtimes G^d$ with an abelian group $G^d$. Here the Clifford complex is decomposed into the dicrect sum over $G^d$--elements, each corresponding to a subcomplex. For any $g \in G$ the $g$--th subcomplex is computed explicitly. Namely, one of the differentials of the mixed complex vanishes after the certain quasi-isomorphism. This is a place where it is important that $G^d$ is abelian, however works similarly if $g$ is diagonalized. The vanishing of this differential constitutes restriction to the (linear) fixed locus of $g$ for the fibre variables.  This is exactly the place from where the operator $\rmH_{f,g}$ comes. This step is still applicable for a non--abelian symmetry group with the restriction to the fixed locus being understood linearly.

Subtle point of the three steps above is the $G$--action. The quasi--isomorphism above are only proved to be $G$--equivariant when $G = G^d$ is abelian. This is the place where can not just compute the $G$--action by the definition and need to derive it from the braided commutativity property.

\section{Multiplication table}\label{section: multiplication table}
The aim of this section is to provide the technique to compute any product in $\A'_{f,G}$ explicitly. 
Note that $\A'_{f,G}\subset \A'_{f,G_f}$ is a subalgebra. Therefore in order to compute the $\cup$--product in $\A'_{f,G}$ we may use the elements $\xi_u$ with $u\in G_f$. In what follows we show that it's enough to know the $\cup$--product of $\xi_u$ with $u$ being either a special cycle or diagonal.

\subsection{Disjoint cycle decomposition}
Let $u  \in G$ be decomposed into the product of non--intersecting cycles $u = \prod_{a=0}^k (\sigma_a,g_a)$ with $\sigma_0 = \id$. In particular, we assume $I_{g_a}^c \subseteq I_{\sigma_a}^c$ for all $a \neq 0$.
We call it \textit{disjoint cycle decomposition of $(\sigma,g)$}.

For every $1 \le a \le k$ let $\phi_a := \det(g_a)$. Define $u_{a}^2 := (\id,\phi_a)$ with $\phi_a$ acting on the first coordinate of $I_{\sigma_a}^c$. 
Let $u_a^1$ be s.t. $(\sigma_a,g_a) = u_a^1 \cdot u_a^2$ and $u_0^1 := \id$, $u_0^2:= g_0$.

We have $(\sigma,g) = \prod_{a=0}^k u_a^1 \cdot u_a^2$, calling it \textit{special} cycle decomposition of $(\sigma,g)$.

\begin{remark}
    The choice of an element $u_a^1$ is not unique --- one could choose the coordinate number, on which it acts, to be different --- any one from $I_{\sigma_a}^c$. The choice we fix makes the $\cup$--product formulae simplier.
\end{remark}

\subsection{Cycle decomposition of a pair $(\alpha,g), (\beta,h) \in G$}\label{section: generalized cycle decomposition}
For any pair of elements $(\alpha,g), (\beta,h) \in G$ we build up a special cycle decomposition of both elements respecting geometry of both. 

Let $(\alpha, g) = \prod (\alpha_\bullet,g_\bullet)$ and $(\beta, h) = \prod (\beta_\bullet,h_\bullet)$ be the disjoint cycle decompositions. For every $k,l \neq 0$ decompose further every cycle $\alpha_k$, $\beta_l$ respectively into the product of non--disjoint cycles $\alpha_k = \alpha_{k,1}\cdot\dots\cdot\alpha_{k,i_k}$, $\beta_l = \beta_{l,1}\cdot\dots\cdot\beta_{l,j_l}$, so that for all possible indices $k,i, l,j$ we have:
\[
    I_{\alpha_{k,i}}^c \cap I_{\beta_{l,j}}^c \text{ is either of length } 0 \text{ or of length } 1 \text{ or coincides with } I_{\beta_{l,j}}^c
\]
In this way we get that either $\alpha_{k,i}\beta_{l,j} = \id$, $\alpha_{k,i}\beta_{l,j}$ is a cycle or $\alpha_{k,i}, \beta_{l,j}$ act on the different set of the variables and therefore commute. 

Respective to this decomposition of $\alpha$ and $\beta$ let $g = g_0 \cdot \prod g_{k,i}$ and $h = h_0 \cdot \prod h_{l,j}$ be decomposed into $G^d$--elements as in section above.
Now we have
\[
    (\alpha, g) = (\id,g_0) \cdot \prod_{k \ge 1, i} (\alpha_{k,i},g_{k,i}), \quad (\beta, h) = (\id,h_0) \cdot \prod_{l \ge 1, j} (\alpha_{l,j},h_{l,j}).
\]
We call this \textit{cycle decomposition of a pair} $(\alpha,g), (\beta,h)$. In what follows we also need to consider its \textit{special} version. Namely, as in the section above assume $\phi_\bullet := \det(g_\bullet)$, $\psi_\bullet := \det(h_\bullet)$ (where $\bullet$ stands for the double index notation above) acting on the first coordinate of $g_\bullet$, $h_\bullet$ respectively. As we had before, we get
\[
    (\alpha,g) = \prod u_\bullet^1 \cdot u_\bullet^2, \quad (\beta,h) = \prod v_\bullet^1 \cdot v_\bullet^2
\]
with $u_\bullet^2,v_\bullet^2 \in G_f^d$ and $u_\bullet^1,v_\bullet^1$ --- special.

\begin{example}
    Consider $\alpha = (1,2,3,4)(5,6)$, $g = t_1^at_2^bt_4^ct_6^d$ and $\beta = (1,2)(4,5)$, $h = \id$. 
    
    We get $\alpha_1 = (1,2,3,4)$, $\alpha_2 = (5,6)$ and $\beta_1 = (1,2)$, $\beta_2 = (4,5)$. Now decompose $\alpha_1 = (1,2)(2,3,4) =: \alpha_{1,1}\cdot\alpha_{1,2}$. This can be depicted as
     \begin{alignat}{5}
         &\alpha =&(&1,2)(2&,3, &4) (5&,6)&
         \\
         &\beta =&(&1,2)& &(4,5).&&
     \end{alignat}
    Set $g_1 := t_1^at_2^bt_4^c$, $g_2 = t_6^d$ and further $g_{1,1} := t_1^at_2^{b-e}$, $g_{1,2} := t_2^xt_4^c$, $g_{2,1} = g_2 = t_6^d$ for any $e \in \ZZ$. Respectively $(\alpha,g) = ((1,2),t_1^at_2^{b-e})((2,3,4),t_2^xt_4^c)((5,6),t_6^d)$. 
    
    We have $\phi_{1,1} = \zeta_n^{a+b-e}$, $\phi_{1,2} = \zeta_n^{-c-e}$, $\phi_{2,1} = \zeta_n^{-d}$ and therefore
    \begin{align*}
        & u^1_{1,1} = ((1,2),t_1^{e-b}t_2^{b-e}), u^2_{1,1} = (\id,t_1^{a+b-e}), \quad u^1_{1,2} = ((2,3,4),t_2^{-c}t_4^c), u^2_{1,2} = (\id,t_2^{e+c})
        \\
        & u^1_{2,1} = ((5,6),t_5^{-d}t_6^{d}), u^2_{2,1} = (\id,t_5^{d}).
    \end{align*}
    We have $\psi_{1,1} = 1$ and therefore
    \begin{align*}
        & v^1_{1,1} = ((1,2),\id), \ v^2_{1,1} = (\id,\id), \quad v^1_{2,1} = ((4,5),\id), \ v^2_{2,1} = (\id,\id).
    \end{align*}

\end{example}

\subsection{The strategy}\label{section: the strategy}

Given two $(\alpha,g),(\beta,h) \in G$ in order to compute $\xi_{(\alpha,g)} \cup \xi_{(\beta,h)}$ we follow the steps below.

\subsubsection{Basis vectors decomposition}
Let $(\alpha,g) = \prod_{a=0}^{p_1} (\alpha_a,g_a)$ and $(\beta,h) = \prod_{b=0}^{p_2} (\beta_b,h_b)$ be the cycle decompositions as in Section~\ref{section: generalized cycle decomposition} with the special cycle decompostions $(\alpha,g) = \prod u_a^1 \cdot u_a^2$ and $(\beta,h) = \prod v_a^1 \cdot v_a^2$.

\begin{itemize}
    
    \item By Proposition~\ref{prop: product of special with diagonal} we will compute 
    \[
        \xi_{u_a^1} \cup \xi_{u_a^2} = \xi_{(\alpha_a,g_a)} \text{ and } \xi_{v_b^1} \cup \xi_{v_b^2} = \xi_{(\beta_b,h_b)},
    \]
    for all values of the indices $a$ and $b$.

    \item
    By Proposition~\ref{prop: product of non-intersecting elements} and Eq~\eqref{eq: G--action on id} there is $c_1 \in \CC^*$, s.t. 
    \[
        \left( \xi_{u_0^1}\cup \xi_{u_0^2} \right) \cup \dots \cup \left( \xi_{u_{p}^1}\cup \xi_{u_{p}^2} \right) = \pm \left( \bigcup_{a=0}^{p_1} \xi_{ u_{a}^1} \right) \cup \left( \bigcup_{a=0}^{p_1} \xi_{ u_{a}^2} \right) = c_1 \xi_{(\alpha,g)},
    \]
    for special $G$--elements $u_\bullet^2$ and some $ u_\bullet^1 \in G^d$ with non--intersecting domain. Similarly there is $c_2 \in \CC^*$, s.t.
    \[
        \left( \xi_{v_0^1}\cup \xi_{v_0^2} \right) \cup \dots \cup \left( \xi_{v_{q}^1}\cup \xi_{v_{q}^2} \right) = c_2 \xi_{(\beta,h)}.
    \]
    for special $G$--elements $ v_\bullet^2$ and some $v_\bullet^1 \in G^d$ with non--intersecting domain.
\end{itemize}

\subsubsection{Multiplication}
    Now the product $\xi_{(\alpha,g)} \cup \xi_{(\beta,h)}$ can be computed with the help of decomposition above. In particular

    \begin{itemize}
     \item By braided commutativity and Eq.\eqref{eq: G--action on diagonalizable} there exists $c_3 \in \CC^*$, s.t. 
        \begin{align*}
            \xi_{(g,\alpha)} \cup \xi_{(h,\beta)} &= (c_1c_2)^{-1} 
            \left( \xi_{u_0^1}\cup \xi_{u_0^2} \right) \cup \dots 
            \cup \left( \xi_{v_{q}^1}\cup \xi_{v_{q}^2} \right)
            \\
            &= c_3 (c_1c_2)^{-1} 
            \left( \bigcup \xi_{u_a^1}\cup \xi_{v_b^1} \right) \cup \left( \bigcup \xi_{u_a^2}\cup \xi_{v_b^2} \right)
        \end{align*}
     \item The product of $\xi_{ u_\bullet^2}$,$\xi_{ v_\bullet^2}$ -- elements is computed by Proposition~\ref{prop: diagonal elements product}. It gives us an element $[\phi_1]\xi_{(\id,A)}$ with $(\id,A) \in G^d$.
     \item The product of $\xi_{ u_\bullet^1}$,$\xi_{ v_\bullet^1}$ -- elements is computed by Proposition~\ref{prop: product of non-intersecting elements}, Proposition~\ref{prop: cup-product of generalized cycles with inverse}, Proposition~\ref{prop: special cycles intersecting by one index}. It gives us an element $[\phi_2]\xi_{(\alpha\beta,B)}$ with special $(\alpha\beta,B) \in G$.
     \item Finally $[\phi_1]\xi_{(\id,A)}\cup[\phi_2]\xi_{(\alpha\beta,B)}$ is computed by Proposition~\ref{prop: product of special with diagonal}.
    \end{itemize}

\subsection{Diagonal subgroup $G^d$} 
The $g$--sectors $\A'_{f,g}$ for all $g \in G^d$ form a subalgebra of $\A'_{f,G}$. 
The product of $\A'_{f,G^d}$ was computed explicitly in \cite{BT2}. Because the formula for the products $\xi_g \cup \xi_h$, $g,h\in G^d$ is literally the same in our context, we may use the formulae of \cite{BT2} for the structure constants of $\A'_{f,G^d} \subset \A'_{f,G}$. In particular, we make use of the following formulae.

\begin{proposition}[Proposition~9, Lemma~11 and~12 
in \cite{BT2}]\label{prop: diagonal elements product}~
\begin{itemize}
 \item For any $g,h \in G^d \backslash \{\id\}$, s.t. $I_g \cup I_h \cup I_{gh} \neq I_\id$ we have $\xi_g \cup \xi_h = 0$. 
 \item For each $g=(g_1,\dots,g_N)\in G^d\backslash\{\id\}$ 
, we have 
    \begin{equation}
    \xi_{g}\cup\xi_{g^{-1}}=(-1)^{\frac{d_g(d_g-1)}{2}}\cdot \prod_{i \in I_g^c} \frac{1}{g_i - 1} \lfloor H_{g,g^{-1}}\rfloor_\id \xi_\id.
    \end{equation}
    where 
    \begin{equation}
    H_{g,g^{-1}}:=\widetilde{m}_{g,g^{-1}}\det\left(\frac{\partial^2 f}{\partial x_i\partial x_j}\right)_{i,j\in I_g^c}
    \end{equation}
    and $\widetilde{m}_{g,g^{-1}}$ is a constant uniquely determined by the following equation in $\Jac(f)$
    \begin{equation}\label{H and hessians}
    \frac{1}{\mu_{f^g}}\lfloor\hess(f^g)H_{g,g^{-1}}\rfloor=\frac{1}{\mu_{f}}\lfloor\hess(f)\rfloor.
    \end{equation}

\end{itemize}
\end{proposition}

This proposition does not give a formula for the cup product when $g,h$ are arbitrary, however we do not need it in this text referring interested reader to \cite{BT2,BTW16}.

\begin{example}
Let's illustrate the proposition above on the simple example $g = t_k^p$ for some fixed $1 \le k \le N$ and $1 \le p < n$. We have $f^g = \sum_{i\neq k}x_i^n$ and Eq.\eqref{H and hessians} reads
\begin{align*}
    &\frac{\widetilde m}{(n-1)^{N-1}} \Bigg\lfloor n(n-1)x_k^{n-2} \cdot \prod_{\substack{1 \le i \le N \\ i \neq k}} n(n-1) x_i^{n-2} \Bigg\rfloor
    = \frac{1}{(n-1)^N} \left\lfloor \prod_{1 \le i \le N } n(n-1) x_i^{n-2} \right\rfloor
\end{align*}
giving
\[
     H_{g,g^{-1}} = n \cdot x_k^{n-2}.
\]
We have $d_g = 1$, $\age(g) = p/n$. This gives
\begin{align*}
    \xi_{t_k^p} \cup \xi_{t_k^{-p}} &= \frac{\exp(-\pi \sqrt{-1} \cdot \frac{p}{n})}{\exp(\pi \sqrt{-1} \cdot \frac{p}{n})-\exp(-\pi \sqrt{-1} \cdot \frac{p}{n})} \cdot n \cdot \lfloor x_k^{n-2} \rfloor \xi_\id
    \\
    &= \frac{n}{\zeta_k^p- 1} \cdot \lfloor x_k^{n-2} \rfloor \xi_\id.
\end{align*}
\end{example}




\subsection{``Symmetric'' subgroup $G^s$} In this section consider the cup--product of elements $\alpha,\beta \in G^s$. Note that such elements are all special.
\begin{proposition}\label{prop: Clifford multiplications}
    Let $\alpha,\beta \in G^s$ satisfy one of the following conditions
    \begin{description}
     \item[(a)] $\alpha,\beta$ are arbitrary non--intersecting,
     \item[(b)] $\alpha,\beta$ are both cycles and $|I_\alpha^c \cap I_\beta^c| = 1$,
     \item[(c)] $\alpha = \sigma'\sigma''$ being a product of two non--intersecting cycles and $\beta$ is a cycle, s.t. $\alpha\beta$ is a cycle.
    \end{description}
    Then there are numbers $\eps_{\alpha,\beta} \in \{+1,-1\}$, s.t.  $\p_{\tilde \theta_g} \cdot \p_{\tilde \theta_h} = \eps_{\alpha,\beta} \p_{\tilde \theta_{\alpha\beta}}$ holds in $\mathrm{Cl}_N[\p_\theta]$.
\end{proposition}
\begin{proof}
    ~\\
    {\bf Case (a).}
    This follows immediately from our definition of the elements $\p_{\tilde \theta}$. Namely, the cycles of $\alpha\beta$ are either cycles of $\alpha$ or cycles of $\beta$. Thus, the elements $\p_{\widetilde\theta_\alpha} \cdot \p_{\widetilde\theta_\beta}$ and $\p_{\widetilde\theta_{\alpha\beta}}$ as products over cycles has the same multipliers and, hence, coincide up to sign.
    
    In particular, for  $\alpha$ and $\beta$ being nonintersecting cycles $(i_1,\dots,i_k)$ and $(j_1,\dots,j_l)$ respectively we get 
    \[
        \p_{\tilde \theta_\alpha} \cdot \p_{\tilde \theta_\beta} = \begin{cases} \p_{\tilde \theta_{\alpha\beta}} & \quad \text{ if } i_1 < j_1, \\  (-1)^{(k-1)(l-1)} \p_{\tilde \theta_{\alpha\beta}}& \quad \text{ if } i_1 > j_1.\end{cases}.
    \]
    ~\\
    {\bf Case (b).}
    Without loss of generality assume $\alpha = (i_1,\dots,i_k)$, $\beta = (j_1,\dots,j_l)$ and $i_a = j_b$. 
    We show that
  \[
    \eps_{\alpha,\beta} = 
    \begin{cases} (-1)^{(k-1+a+b)(l-1)} \xi_{\alpha\beta} & \quad \text{ if } i_1 \le j_1, 
    \\  
    (-1)^{(a+b)(k-1)} \xi_{\alpha\beta} & \quad \text{ if } i_1 > j_1.\end{cases}
  \]
  Assume that $i_1 \le j_1$. Basic computations in $\Cl_N$ give
  \begin{align*}
    &\p_{\widetilde \theta_\alpha} \cdot \p_{\widetilde \theta_\beta} =  \sum_{m=1}^k (-1)^{m-1} \p_{\theta_{i_1}} \cdot \dots \widehat{m} \dots \cdot \p_{\theta_{i_k}} \cdot \sum_{m'=1}^l (-1)^{m'-1} \p_{\theta_{j_1}} \cdot \dots \widehat{m'} \dots \cdot \p_{\theta_{j_l}}
    \\
    &= (-1)^{(k-a)(l-1)} (-1)^{a-1}\sum_{m'=1}^l (-1)^{m'-1}  \p_{\theta_{i_1}} \cdot \dots (\p_{\theta_{j_1}} \cdot \dots \widehat{m'} \dots \cdot \p_{\theta_{j_l}} ) \dots \cdot \p_{\theta_{i_k}}
    \\
    & + (-1)^{(k-a)(l-1)}  (-1)^{b-1} \sum_{m=1}^{a} (-1)^{m-1} \p_{\theta_{i_1}} \cdot \dots \widehat{m} \dots (\p_{\theta_{j_1}} \cdot \dots \widehat{b} \dots \cdot \p_{\theta_{j_l}}) \dots \cdot \p_{\theta_{i_k}}
    \\
    & + (-1)^{(k-a-1)(l-1)} (-1)^{b-1} \sum_{m=a+1}^k (-1)^{m-1} \p_{\theta_{i_1}} \cdot  \dots (\p_{\theta_{j_1}} \cdot \dots \widehat{b} \dots \cdot \p_{\theta_{j_l}}) \dots \widehat{m} \dots \cdot \p_{\theta_{i_k}}
    \\
    &= 
    (-1)^{(k-a)(l-1)}  (-1)^{b-1} \sum_{m=1}^{a} (-1)^{m-1} \p_{\theta_{i_1}} \cdot \dots \widehat{m} \dots (\p_{\theta_{j_1}} \cdot \dots \widehat{b} \dots \cdot \p_{\theta_{j_l}}) \dots \cdot \p_{\theta_{i_k}}
    \\
    & + 
    (-1)^{(k-a)(l-1)} \sum_{m'=1}^l (-1)^{(m'+a-1)-1}  \p_{\theta_{i_1}} \cdot \dots (\p_{\theta_{j_1}} \cdot \dots \widehat{m'} \dots \cdot \p_{\theta_{j_l}} ) \dots \cdot \p_{\theta_{i_k}}
    \\
    & + (-1)^{(k-a)(l-1)} (-1)^{b-1} \sum_{m=a+1}^k (-1)^{(m+l-1)-1} \p_{\theta_{i_1}} \cdot  \dots (\p_{\theta_{j_1}} \cdot \dots \widehat{b} \dots \cdot \p_{\theta_{j_l}}) \dots \widehat{m} \dots \cdot \p_{\theta_{i_k}}.
  \end{align*}
  Finally we should reorder the terms in the brackets so that $\p_{\theta_{j_b}}$ appears on the first place. This introduces additional factor of $(-1)^{(b-1)(l-b)}$ to all three summands above and additional factor of $(-1)^{b-1}$ to the second summand only (because it involves a summation over the affected elements). 
  This completes the proof for $i_1 < j_1$. 
  The case $i_1 > j_1$ is treated completely similarly giving $\p_{\tilde \theta_\alpha} \cdot \p_{\tilde \theta_\beta} = (-1)^{(b-1)(k-1)} (-1)^{(a-1)(k-a-1)} \p_{\tilde \theta_{\alpha\beta}}$.
 ~\\
 {\bf Case (c).}
 Assume $\sigma' = (i_1,\dots,i_k)$, $\sigma'' = (j_1,\dots,j_l)$. It's enough to show the claim for $\beta = (i_a,j_b)$ because any more general $\beta$ can be decomposed into the product of permutations. We have
 \begin{align*}
    \p_{\tilde\theta_\alpha} \cdot \p_{\tilde\theta_{(i_a,j_b)}} 
    &= \left( \eps_{\sigma',\sigma''}\p_{\tilde\theta_{\sigma'}} \cdot \p_{\tilde\theta_{\sigma''}} \right) \cdot \p_{\tilde\theta_{(i_a,j_b)}} & \quad \textit{by case (a) above}
    \\
    &= \eps_{\sigma',\sigma''}\p_{\tilde\theta_{\sigma'}} \cdot \eps_{\sigma'',(i_a,j_b)} \p_{\tilde\theta_{\sigma''(i_a,j_b)}} & \quad \textit{by case (b) above}
    \\
    &= \eps_{\sigma',\sigma''} \eps_{\sigma'',(i_a,j_b)} \eps_{\sigma',\sigma'(i_a,j_b)} \p_{\tilde\theta_{\sigma'\sigma''(i_a,j_b)}}
 \end{align*}
 where in the last equality we have used case (b) above again because the pair $\alpha' = \sigma'$, $\beta' = \sigma''(i_a,j_b)$ satisfies its conditions.
\end{proof}

\begin{proposition}\label{prop: simple cup-products}
    Let $\alpha,\beta$ be as in Proposition~\ref{prop: Clifford multiplications} above. We have $\xi_\alpha \cup \xi_\beta = \eps_{\alpha,\beta} \xi_{\alpha\beta}$.
\end{proposition}
\begin{proof}
 We make use of the formula Eq.\eqref{eq: sigma g,h}.
  The elements $\xi_\alpha$ and $\xi_\beta$ are represented by the classes ${\exp(\rmH_{f,\alpha}) \cdot \p_{\tilde \theta_\alpha}}$ and ${\exp(\rmH_{f,\beta}) \cdot \p_{\tilde \theta_\beta}}$ respectively. The elements $\rmH_{f,\alpha}, \rmH_{f,\beta}$ are non-zero whenever $\alpha,\beta$ are not just transpositions. 
  However it turns out that in the cases under consideration we don't need to use the explicit form of these elements. We divide the proof in the cases of Proposition~\ref{prop: Clifford multiplications}.
  ~
  \\
  {\bf Cases (a) and (b).} The length of the element $\p_{\tilde \theta_{\alpha\beta}}$ is equal to the sum of lengths of the elements $\p_{\tilde \theta_{\alpha}}$ and $\p_{\tilde \theta_{\beta}}$. Therefore to find the product $\xi_\alpha \cup \xi_\beta$ by formule Eq.\eqref{eq: sigma g,h} it's enough to consider the coefficient of $\p_{\tilde \theta_{\alpha\beta}}$ in $\LL\left( \p_{\tilde \theta_{\alpha}} \otimes \p_{\tilde \theta_{\beta}} \right)$ because applying $\rmH_{f,\alpha}$, $\rmH_{f,\beta}$ or $\rmH_f$ one reduces the length of the Clifford element by two.
  ~
  \\
  {\bf Case (c)} holds due to the same reasons. Alternatively, it can be deduced following the proof of this case in Proposition~\ref{prop: Clifford multiplications} but by using $\cup$--product of $\xi_\bullet$ elements rather than Clifford elements.
\end{proof}
\begin{corollary}\label{cor: transpositions cup--product commutation}
 For $i < j < k$ the product of transpositions reads
  \begin{align}
  &\xi_{(i,j)} \cup \xi_{(j,k)} = \xi_{(i,j,k)}, \quad \xi_{(i,k)} \cup \xi_{(j,k)} = -\xi_{(i,k,j)}, \quad \xi_{(j,k)} \cup \xi_{(i,j)} = \xi_{(i,k,j)},
  \\
  &\xi_{(i,j)} \cup \xi_{(i,k)} = -\xi_{(i,k,j)}, \quad \xi_{(i,k)} \cup \xi_{(i,j)} = -\xi_{(i,j,k)}, \quad \xi_{(j,k)} \cup \xi_{(i,k)} = -\xi_{(i,j,k)}.
  \end{align}
  In particular, $\xi_{(i,j)} \cup \xi_{(j,k)} = - \xi_{(j,k)} \cup \xi_{(i,k)}$.
\end{corollary}
\begin{proof}
    This follows immediately via $\varepsilon_{\bullet,\bullet}$ from the proposition above.
\end{proof}

For $\alpha,\beta \in G^s$ being two nonintersecting cycles $(i_1,\dots,i_k)$ and $(j_1,\dots,j_l)$ respectively, we have by Proposition~\ref{prop: simple cup-products}
\[
  \xi_\alpha \cup \xi_\beta = \begin{cases} \xi_{\alpha\beta}& \quad \text{ if } i_1 < j_1, \\  (-1)^{(k-1)(l-1)}\xi_{\alpha\beta}& \quad \text{ if } i_1 > j_1.\end{cases}
\] 

\begin{remark}
    The $\pm$ sign in the product of transpositions above can be understood in the following way. Associate to a trasposition $(i,j)$ the graph $\Gamma_{(i,j)}$ with two vertices indexed by $i$ and $j$ and an edge between them oriented from the smaller index vertex to the greater index vertice. 
    
    For any two transpositions $\tau_1$ and $\tau_2$ let the graphs $\Gamma_{\tau_1}$ and $\Gamma_{\tau_2}$ have a vertex of the same index. Compose a new graph $\Gamma'$ by gluing $\Gamma_{\tau_1}$ and $\Gamma_{\tau_2}$ by this vertex into a valence two vertex. We have $\eps_{\tau_1,\tau_2} = +1$ if one of its edges is oriented to and the other out and $\eps=-1$ if both edges either are oriented out or both oriented in and $\xi_{\tau_1} \cup \xi_{\tau_2} = \eps_{\tau_1,\tau_2} \xi_{\tau_1 \tau_2}$.
\end{remark}

\begin{proposition}\label{prop: product of non-intersecting elements}
Let $u,v\in G$ be non--intersecting elements. Then
  \[
  \xi_u \cup \xi_v = \eps_{u,v} \xi_{uv}.
  \]
\end{proposition}
\begin{proof}
    This is immediate by Eq.\eqref{eq: sigma g,h}.
\end{proof}

\begin{proposition}\label{prop: cup-product of cycles with inverse}
Let $\sigma \in G^s$ be a cycle $\sigma = (i_1,\dots,i_k)$.
We have
\[
    \xi_\sigma \cup \xi_{\sigma^{-1}} = (-n)^{k-1} \left\lfloor \Phi_{i_1,\dots,i_k}(\bx) \right \rfloor \xi_\id.
\]
where 
\[ 
    \Phi_{i_1,\dots,i_k}(\bx) := \sum_{\substack{0 \le a_1,\dots,a_k \le n \\ \sum a_\bullet = (k-1)(n-2)}} x_{i_1}^{a_1}\dots x_{i_k}^{a_k}.
\]
\end{proposition}

We first consider the simple case:

\begin{lemma}\label{square of transposition}
For any $i,j \le N$ we have 
$$
\xi_{(i,j)} \cup \xi_{(i,j)} = -n \left\lfloor \Phi_{ij} \right\rfloor \cdot \xi_\id.
$$
\end{lemma}
\begin{proof}
Assume $i < j$. 
We have $\rmH_{f,(i,j)} = 0$ and in order to compute the product we should compute
\begin{align*}
 \Upsilon\Big( \rmH_f &(\bx,(i,j)(\bx),\bx) \cdot \left( (\p_{\theta_i} - \p_{\theta_j}) \otimes (\p_{\theta_j} - \p_{\theta_j}) \right)\Big)
  \\
  & =  \Upsilon\Big(\sum_{a+b = n-2} (a+1) x_i^a x_j^b \theta_i \otimes \theta_i \cdot (\p_{\theta_i} \otimes \p_{\theta_i}) + \sum_{a+b = n-2} (a+1) x_j^a x_i^b \theta_j \otimes \theta_i \cdot (\p_{\theta_j} \otimes \p_{\theta_j})\Big)
  \\
  & = - \sum_{a+b = n-2} (a+1)x_i^a x_j^b- \sum_{a+b = n-2} (a+1) x_j^a x_i^b = -n \sum_{a+b = n-2} x_j^a x_i^b .
\end{align*}
We get the same for $i > j$ because of two sign changes on the right hand side of the formulae.
\end{proof}

\begin{proof}[Proof of Proposition~\ref{prop: cup-product of cycles with inverse}]

Now assume $\sigma \in G^s$ is a cycle $\sigma = (i_1,\dots,i_k)$. 
We can decompose $\sigma = \tau_{1} \tau_{2}\cdot \dots \cdot \tau_{k-1}$ for $\tau_a = (i_a,i_{a+1})$.
By Proposition~\ref{prop: simple cup-products} we hase
\begin{align*}
    & \xi_\sigma = \xi_{\tau_{1}} \cup \xi_{\tau_{2}} \cup \dots \cup \xi_{\tau_{k-1}},
    \\
    & \xi_{\sigma^{-1}} = \xi_{\tau_{k-1}} \cup \dots \cup \xi_{\tau_{2}} \cup \xi_{\tau_{1}}.
\end{align*}
Therefore we have
\begin{align*}
    \xi_\sigma \cup \xi_{\sigma^{-1}} 
    &= \xi_{\tau_{1}} \cup \xi_{\tau_{2}} \cup \dots \cup \left(  \xi_{\tau_{k-1}} \cup \xi_{\tau_{k-1}} \right) \cup \dots \cup \xi_{\tau_{2}} \cup \xi_{\tau_{1}} && \text{\it (above lemma)}
    \\
    &= \xi_{\tau_{1}} \cup \xi_{\tau_{2}} \cup \dots \cup \left(  \xi_{\tau_{k-2}} \cup -n \lfloor \Phi_{i_{k-1},i_k} \rfloor \xi_\id \cup \xi_{\tau_{k-2}} \right) \cup \dots \cup \xi_{\tau_{2}} \cup \xi_{\tau_{1}} && \text{\it (Eq~\eqref{eq: braided commutativity} and \eqref{eq: G--action on diagonalizable})}
    \\
    &= \lfloor -n\Phi_{i_{1},i_k} \rfloor \xi_\id\cup \xi_{\tau_{1}} \cup \xi_{\tau_{2}} \cup \dots \cup \left(  \xi_{\tau_{k-2}} \cup \xi_{\tau_{k-2}} \right) \cup \dots \cup \xi_{\tau_{2}} \cup \xi_{\tau_{1}}
    \\
    &= \lfloor -n \Phi_{i_{1},i_k} \rfloor \xi_\id\cup \lfloor -n\Phi_{i_{1},i_{k-1}} \rfloor \xi_\id \cup \dots \cup \lfloor -n\Phi_{i_1,i_2} \rfloor \xi_\id 
    \\
    &= (-n)^{k-1}\sum_{\substack{0 \le a_1,\dots,a_k \le n \\ \sum a_\bullet = (k-1)(n-2)}} \left\lfloor x_{i_1}^{a_1}\dots x_{i_k}^{a_k} \right\rfloor \xi_\id.
\end{align*}
\end{proof}

\begin{corollary}
    Let $\alpha=(i_1,\dots,i_k,j_1,\dots,j_q)$ and $\beta = (j_p,\dots,j_{p+l},j_{q-1},\dots,j_1,j_q)$. 
    We have
    \[
        \xi_\alpha \cup \xi_\beta = 
        \eps_{\alpha',\alpha''}\eps_{\beta',\beta''} \eps_{\alpha',\beta'} (-n)^{q-1} 
        \left\lfloor \Phi_{i_1,j_2,\dots,j_q} \right\rfloor \xi_{\alpha\beta},
    \]
    for $\alpha' = (i_1,\dots,i_k,j_1)$ and $\beta' = (j_q,j_p,\dots,j_{p+l})$.
\end{corollary}
\begin{proof}
    We can decompose $\alpha = \alpha'\alpha'' = (i_1,\dots,i_k,j_1)(j_1,\dots,j_q)$ and $\beta = (\alpha'')^{-1}\beta' = (j_q,\dots,j_1)(j_q,j_p,\dots,j_{p+l})$. By using Proposition~\ref{prop: simple cup-products} and proposition above we get:
    \begin{align*}
        \xi_\alpha \cup \xi_\beta &= \eps_{\alpha',\alpha''}\eps_{\beta',\beta''} \xi_{\alpha'}\cup\xi_{\alpha''} \cup \xi_{(\alpha'')^{-1}}\cup\xi_{\beta'} 
        \\
        &= \eps_{\alpha',\alpha''}\eps_{\beta',\beta''} \xi_{\alpha'} \cup (-n)^{q-1} \left\lfloor \Phi_{j_1,\dots,j_q} \right\rfloor \xi_\id \cup \xi_{\beta'}
        \\
        &= \eps_{\alpha',\alpha''}\eps_{\beta',\beta''} (-n)^{q-1} \left\lfloor \Phi_{i_1,j_2,\dots,j_q} \right\rfloor \xi_\id \cup \eps_{\alpha',\beta'} \xi_{\alpha'\beta'}.
    \end{align*}
    by using Eq~\eqref{eq: G--action on diagonalizable} for the last equality.
\end{proof}

Two propositions above allow us to compute the cup--product of any $u,v \in G^s$ that are special by definition. Let's illustrate it with the following example.

\begin{example}
    Let $\alpha = (i,j)(k,l)$, $\beta = (i,k)(j,l)$ and $\gamma := \alpha\beta = (i,l)(j,k)$ with ${i < j < k < l}$. We have
    \begin{align}
        \xi_\alpha \cup \xi_\beta & = \xi_{(i,j)} \cup \xi_{(k,l)} \cup \xi_{(i,k)}\cup \xi_{(j,l)} 
        && \textit{(Proposition~\ref{prop: product of non-intersecting elements})}
        \\
        & 
        = - \xi_{(i,j)} \cup \xi_{(k,l)} \cup \xi_{(j,l)} \cup \xi_{(i,k)}
        && \textit{(Eq~\eqref{eq: braided commutativity} and ~\eqref{eq: G--action on diagonalizable})}
        \\
        & = \xi_{(i,j)} \cup \xi_{(j,k,l)} \cup \xi_{(i,k)} = \xi_{(i,j,k,l)} \cup \xi_{(i,k)}
        && \textit{(Proposition~\ref{prop: simple cup-products})}
        \\
        & = \xi_{(i,l)} \cup \xi_{(i,j,k)} \cup \xi_{(i,k)} = 
        - \xi_{(i,l)} \cup \xi_{(j,k)} \cup \xi_{(i,k)} \cup \xi_{(i,k)}
        \\
        & = \xi_{(i,l)(j,k)} \cup n \lfloor \Phi_{ik} \rfloor \xi_{\id}
        && \textit{(Proposition~\ref{prop: cup-product of cycles with inverse})}
        \\
        & = n \lfloor \Phi_{lj} \rfloor \xi_\gamma.
    \end{align}
\end{example}
Important case of the corollary above is the following. Let $\sigma \in G^s$ be a cycle $\sigma = (i_1,\dots,i_k)$. Then
\[
    \xi_{(i_a,i_b)} \cup \xi_\sigma = - \eps_{(i_a,i_b),\sigma(i_a,i_b)} n \left\lfloor \Phi_{i_a,i_b} \right\rfloor\xi_{(i_a,i_b)\sigma}.
\]


\begin{corollary}\label{cor: generators}
The algebra $\A'_{f,G^s}$ is generated over $\CC$ by the elements $\xi_{(i,j)}$ and $\lfloor x_i \rfloor \xi_\id$.
\end{corollary}
\begin{proof}
It is sufficient to generate $\xi_w$ for all $w\in G^s$. An element $w\in G^s$ could be written in a form $w=\tau_{i_1}\ldots \tau_{i_m}$, for the transpositions $\tau_{i_1} := (i_1,i_1+1)$, where all $\tau_\bullet$ assumed together satisfy the conditions of Proposition~\ref{prop: Clifford multiplications}. Propositions \ref{prop: simple cup-products} and \ref{prop: cup-product of cycles with inverse} give the statement.
\end{proof}

\section{Products in the mixed sectors}\label{section: products in mixed sectors}
In this section we consider the cup products of more general elements.

\begin{proposition}\label{prop: special cycles intersecting by one index}
    Let $u,v \in G$ be special cycles with $u = (i_1,\dots,i_k) g$, $v = (j_1,\dots,j_l) h$ for some $g, h \in G^d$ and $i_a = j_b$ for some $a,b$. 
    Then we have
    \[
        \xi_u \cup \xi_v = 
        \begin{cases}
            \eps_{u,v} \widetilde h_{j_b} \xi_{u\cdot v} \quad \text{ if } i_1 < j_1,
            \\
            \eps_{u,v} \widetilde g_{i_a} \xi_{u\cdot v} \quad \text{ if } i_1 \ge j_1
        \end{cases}        
    \]
    for $\eps_{u,v}$ as in Proposition~\ref{prop: Clifford multiplications}. 
    
    Note that the case $i_1 = j_1$ above can only happen if $a=b=1$. 

\end{proposition}
\begin{proof}
    The elements $\xi_u$ and $\xi_v$ are represented by $\exp(\rmH_{f,u}) \p_{\widetilde \theta_u}$ and $\exp(\rmH_{f,v}) \p_{\widetilde \theta_v}$ respectively. 
    The product $(i_1,\dots,i_k)(j_1,\dots,j_l)$ is a length $k+l-1$ cycle and the element $\xi_{uv}$ is represented by $\exp(\rmH_{f,uv}) \p_{\widetilde \theta_{uv}}$. 
    The Clifford elements $\p_{\widetilde \theta_u}$, $\p_{\widetilde \theta_v}$ and $\p_{\widetilde \theta_{uv}}$ have degrees $k-1$, $l-1$ and $k+l-2$ respectively.
    By the same reasoning as in Proposition~\ref{prop: simple cup-products} we have $\xi_u \cup \xi_v = c \cdot \xi_{uv}$ where $c \in \CC$ is s.t. in Clifford algebra holds $\p_{\widetilde \theta_u} \cdot \p_{\widetilde \theta_v} = c \cdot \p_{\widetilde \theta_{uv}}$.     

    In order to find the constant $c$ it's enough to assume just one summand on the RHS.
    We have
    \begin{align}
        \p_{\widetilde \theta_u} &= (-1)^{a-1} \widetilde g_{i_a} \p_{\theta_{i_1}} \cdots \widehat a \cdots \p_{\theta_{i_k}} + \dots,
        \\
        \p_{\widetilde \theta_v} &= (-1)^{b-1} \widetilde h_{j_b} \p_{\theta_{j_1}} \cdots \widehat b \cdots \p_{\theta_{j_l}} + \dots.
    \end{align}
    Expanding $\p_{\widetilde \theta_{uv}}$ in the same way as above we should take care of the sign and also constant multiples originating from scaling part of $uv$. 
    Note that we do not need to take care of the sign --- it was essentially computed in the previous section to be $\eps_{u,v}$. The rest requires case by case study.

    \noindent
    {\bf Case 1: $i_1 < j_1$}. We have 
    \[
        \p_{\widetilde \theta_{uv}} = (-1)^{a-1} g_{i_1}\cdots g_{i_{a-1}} 
        \p_{\theta_{i_1}}\cdots \p_{\theta_{i_{a-1}}}\p_{\theta_{j_{b+1}}}\cdots\p_{\theta_{j_l}}\p_{\theta_{j_1}}\cdots \p_{\theta_{j_{b-1}}} \p_{\theta_{i_{a+1}}} \cdots \p_{\theta_{i_k}} + \dots .
    \]
    
    \noindent
    {\bf Case 2: $i_1 > j_1$}. By using the fact that $u$ is special we have
    \[
        \p_{\widetilde \theta_{uv}} = (-1)^{k+b} h_{j_1}\cdots h_{j_{b-1}} \p_{\theta_{j_1}}\cdots \p_{\theta_{j_{b-1}}} \p_{\theta_{i_{a+1}}} \cdots \p_{\theta_{i_k}}\p_{\theta_{i_1}}\cdots \p_{\theta_{i_{a-1}}}\p_{\theta_{j_{b+1}}}\cdots\p_{\theta_{j_l}} + \dots .
    \]
    
    \noindent
    {\bf Case 3: $a = b = 1$}. In this case we have 
    \[
        \p_{\widetilde \theta_{uv}} = \p_{\theta_{j_2}}\cdots \p_{\theta_{j_l}} \p_{\theta_{i_2}} \cdots \p_{\theta_{i_k}} + \dots .
    \]
    
    This concludes the proof.
\end{proof}

\begin{proposition}\label{prop: cup-product of generalized cycles with inverse}
    Let $g = t_{i_1}^{d_1}\cdots t_{i_k}^{d_k} \in G$ be special and $\sigma = (i_1,\dots,i_k)$. 
    
    We have
    \[
        \xi_{(\sigma,g)} \cup \xi_{(\sigma,g)^{-1}} = (-n)^{k-1} \left\lfloor \Phi^{(D)}_{i_1,\dots,i_k}(\bx) \right\rfloor \xi_\id.
    \]
    for $D = -d_1-\dots - d_{k-1}$ and 
    \[
     \Phi^{(D)}_{i_1,\dots,i_k} := \Phi_{i_1,\dots,i_k}(\zeta^{-D}x_{i_1},x_{i_2},\dots,x_{i_k}).
    \]
\end{proposition}

We first proof the following lemma.

\begin{lemma}\label{square of mixed transposition}
For $i,j\in\{1,\ldots,N\}$, $i<j$ and $d_1, d_2\in\ZZ$ denote $u_1 := (i,j) \cdot t_i^{d_1}t_j^{-d_1}$, $u_2 := (i,j) \cdot t_i^{d_2}t_j^{-d_2}$. 
We have
\begin{align}
 \xi_{u_1} &\cup\xi_{u_2}
 =
 \begin{cases}
        0 \quad & \text{ if } d_1\not\equiv d_2 \mod n,
        \\
      \left\lfloor\Phi_{ij}^{(d_1)}\right\rfloor\xi_{\id} \quad & \text{ if } d_1\equiv d_2 \mod n.        
        \end{cases}
\end{align}
\end{lemma}
\begin{proof}
We have $\widetilde\xi_{u_1}=\p_{\theta_j}-\zeta^{d_1}\p_{\theta_i}$ and $\widetilde\xi_{u_2}=\p_{\theta_j}-\zeta^{d_2}\p_{\theta_i}$.  
Since for both elements there is only one--dimensional eigenspace with the eigenvalue different from $1$, the elements $\rmH_{f,u_1}$ and $\rmH_{f,u_2}$ vanish. 
We thus have 
\begin{align*}
\xi_{u_1} &\cup\xi_{u_2}=-\left\lfloor \rmH_{f,j}+\zeta^{d_1+d_2}\rmH_{f,i}\right\rfloor\xi_{t_i^{d_1-d_2}t_j^{-d_1+d_2}}
 \\
 =- & \left\lfloor\sum_{a+b=n-2} (a+1)\zeta^{-d_1b}x_j^ax_i^b+\zeta^{d_1+d_2}\sum_{a+b=n-2} (a+1)\zeta^{d_1b}x_i^ax_j^b\right\rfloor\xi_{t_i^{d_1-d_2}t_j^{-d_1+d_2}}
 \\
 =- & \left\lfloor\sum_{a+b=n-2} \left((a+1)\zeta^{-d_1b}+(b+1)\zeta^{d_1a+d_1+d_2}\right)x_j^ax_i^b\right\rfloor\xi_{t_i^{d_1-d_2}t_j^{-d_1+d_2}}
  \\
 =- & \left\lfloor\sum_{a+b=n-2} \left((a+1)\zeta^{-d_1b}+(b+1)\zeta^{-d_1(b+1)+d_2}\right)x_j^ax_i^b\right\rfloor\xi_{t_i^{d_1-d_2}t_j^{-d_1+d_2}}.
\end{align*}
If $d_1\not\equiv d_2 \mod n$ the fixed set of $t_i^{d_1-d_2}t_j^{-d_1+d_2}$ is given by $x_i=x_j=0$ and the above expression vanishes. If $d_1\equiv d_2 \mod n$ we can further rewrite the above formula as
\begin{align*}
& - \left\lfloor\sum_{a+b=n-2} \zeta^{-d_1b}\left((a+1)+(b+1)\right)x_j^ax_i^b\right\rfloor\xi_\id
\\
& \quad\quad =- n\left\lfloor \sum_{a+b=n-2} \zeta^{-d_1b}x_j^ax_i^b\right\rfloor\xi_\id=\left\lfloor \Phi_{ij}^{(d_1)}\right\rfloor\xi_{\id}.
\end{align*}


\end{proof}

\begin{proof}[Proof of Proposition~\ref{prop: cup-product of generalized cycles with inverse}]

Consider the general case $(\sigma,g) = \sigma  \cdot t_{i_1}^{d_1} \cdots t_{i_k}^{d_k} $. Assume first that $i_1 < i_2 < \dots < i_k$.
Following the same lines as in Proposition~\ref{prop: cup-product of cycles with inverse} set
\[
    \alpha_1 = (i_1,i_2) t_{i_1}^{d_1}t_{i_2}^{-d_1}, \dots, \alpha_{k-1} = (i_{k-1},i_k) t_{i_{k-1}}^{d_{k-1}}t_{i_k}^{-d_{k-1}}
\]
so that we get
\begin{align}
    \alpha_1 \cdots \alpha_{k-1} &= (i_1,i_2) t_{i_1}^{d_1}t_{i_2}^{-d_1} \cdot (i_2,i_3) t_{i_2}^{d_2}t_{i_3}^{-d_2} \cdots (i_{k-1},i_k) t_{i_{k-1}}^{d_{k-1}}t_{i_k}^{-d_{k-1}}
    \\
    &= (i_1,i_2) (i_2,i_3) \cdot t_{i_1}^{d_1} t_{i_2}^{d_2} t_{i_3}^{-d_1-d_2} \cdots (i_{k-1},i_k) t_{i_{k-1}}^{d_{k-1}}t_{i_k}^{-d_{k-1}}
    \\
    &= (i_1,i_2) (i_2,i_3) \dots (i_{k-1},i_k) \cdot t_{i_1}^{d_1} t_{i_2}^{d_2} \cdots  t_{i_{k-1}}^{d_{k-1}}t_{i_k}^{-d_1- \cdots -d_{k-1}} = (\sigma,g).
\end{align}
The last equality holds because $(\sigma,g)$ is special.

We have
\begin{align}
    \xi_{(\sigma,g)} & = \xi_{\alpha_1} \cup \dots \cup \xi_{\alpha_{k-1}}, \quad \xi_{(\sigma,g)^{-1}} = \xi_{\alpha_{k-1}} \cup \dots \cup \xi_{\alpha_{1}}.
\end{align}
By using lemma above Eq~\eqref{eq: braided commutativity} and Eq~\eqref{eq: G--action on diagonalizable} this gives us
\begin{align}
    \xi_{(\sigma,g)} \cup \xi_{(\sigma,g)^{-1}} &= \xi_{\alpha_1} \cup \dots \cup \xi_{\alpha_{k-1}} \cup \xi_{\alpha_{k-1}} \cup \dots \cup \xi_{\alpha_{1}}
    \\
    &= \Phi_{i_{1}i_k}^{(d_{k-1})} \xi_{\id} \cup \Phi_{i_{1}i_{k-1}}^{(d_{k-2})} \xi_{\id} \cup   \cdots \cup \Phi_{i_{1}i_2}^{(d_{1})} \xi_\id.
\end{align}
Due to the relations in the Jacobian algebra the last expression multiplies exactly to $\Phi^{(D)}_{i_1,\dots,i_k}(\bx)$.

For $(i_1,\dots,i_k)$ not ordered ascending consider the same decomposition of $(\sigma,g)$. 
Due to Proposition~\ref{prop: special cycles intersecting by one index} we have $\xi_{(\sigma,g)} = \eps \cdot \xi_{\alpha_1} \cup \dots \cup \xi_{\alpha_{k-1}}$ for some $\eps \in \{1,-1\}$ and also $\xi_{(\sigma,g)^{-1}} = \eps \cdot \xi_{\alpha_{k-1}} \cup \dots \cup \xi_{\alpha_{1}}$. 
Now we can follows the same lines as above. 
\end{proof}

The next proposition resolves the cup-products $\xi_\alpha \cup \xi_h$ for special elements $\alpha$ and any $h \in G^d$. This is essentially the last step of our strategy (see Section~\ref{section: the strategy}).

\begin{proposition}\label{prop: product of special with diagonal}
Consider a special cycle $(\sigma,g) \in G$ and $(\id,h) \in G$ with $\sigma = (i_1,\dots,i_k)$ and ${I_h^c \subseteq I_{(\sigma, g)}^c}$. We have the following two cases.

\begin{itemize}
 \item $h$ is not special. Then $\xi_{(\sigma,g)} \cup \xi_{(\id,h)} = \xi_{(\id,h)}\cup \xi_{(\sigma,g)} = 0$ if $|I_h^c| \ge 2$. For $I_h^c = \{i_a\}$ we have
    \begin{align*}
        \xi_{(\sigma,g)} \cup \xi_{(\id,h)} & = (-1)^{k-1} \widetilde g_{i_a} \xi_{(\sigma,g)(\id,h)},
        \\
        \xi_{(\id,h)}\cup \xi_{(\sigma,g)} & = \widetilde g_{i_a} \xi_{(\id,h)(\sigma,g)}
    \end{align*}
%
 \item $h$ is special. Then $\xi_{(\sigma,g)} \cup \xi_{(\id,h)} = 0$ if $h$ can't be decomposed into the product of special $G$--elements $h_\alpha$ satisfying $|I_{h_\alpha}^c| = 2$. 
 
 Moreover for $h$, s.t. $ I_{h}^c = \lbrace i_a,i_b \rbrace$ with $h_{i_a}h_{i_b}=1$ we have
 \begin{align*}
  \xi_{(\sigma,g)} \cup \xi_{(\id,h)} &= 
  n \cdot \widetilde{g}_{i_{a}}\widetilde{g}_{i_{b}} \cdot \frac{h_{i_p}}{h_{i_p} - 1} \cdot \lfloor \widetilde x_{i_1}^{n-2} \rfloor \xi_{(\sigma,g)(\id,h)},
  \\
  \xi_{(\id,h)} \cup \xi_{(\sigma,g)} &= 
  n \cdot \widetilde{g}_{i_{a}}\widetilde{g}_{i_{b}} \cdot \frac{1}{h_{i_p} - 1} \cdot \lfloor \widetilde x_{i_1}^{n-2} \rfloor \xi_{(\id,h)(\sigma,g)}
 \end{align*}
 for $p := \min(i_a,i_b)$.
\end{itemize}
\end{proposition}
\begin{proof}
~\\
{\bf Case 1, $h$ is not special}.\\
The vectors $\xi_{(\id,h)}$ and $\xi_{(\sigma,g)}$ are represented by the Clifford elements $\exp(\rmH_{f,(\id,h)}) \cdot \p_{\theta_{(\id,h)}}$ and $\exp(\rmH_{f,(\sigma,g)}) \cdot \p_{\theta_{(\sigma,g)}}$ respectively. 
However the fixed locus of $(\sigma,gh)$ is $0$ and the expansion of both $\rmH_{f,(\id,h)}$ and $\rmH_{f,(\sigma,g)}$ introduces the $\bx$--multiples that vanish in $(\sigma,gh)$--th sector. Due to this reasoning we have 
\begin{align}
    & \xi_{(\sigma,g)} \cup \xi_{(\id,h)} = c \cdot \xi_{\theta_{(\sigma,gh)}}, \quad c \in \CC,
    \\
    \text{ where } & \p_{\theta_{(\sigma,g)}} \cdot \p_{\theta_{(\id,h)}} = c \cdot \p_{\theta_{(\sigma,gh)}}.
\end{align}
The scalar $c$ is only non--zero if $|I^c_h| = 1$. In this case assume $I_h^c = \{h_{i_a}\}$. We have
\begin{align}
    \p_{\theta_{(\sigma,g)}} \cdot \p_{\theta_{(\id,h)}} &= \sum_{q=1}^k (-1)^{q-1} \widetilde g_{i_q} \p_{\theta_{i_1}} \cdot \dots \widehat{q} \dots \cdot \p_{\theta_{i_k}} \cdot \p_{\theta_{i_a}}.
\end{align}
The sum obtained gives only one non--zero summand for the index $q=a$. In order to comply with our basis choice we have to further commute $\p_{\theta_{i_a}}$ on the $a$--th place what introduces the additional sign $(-1)^{k-a}$.

The product $\xi_{(\id,h)}\cup \xi_{(\sigma,g)}$ is treated completely analogously.
~\\
{\bf Case 2, $h$ is special}.\\
Let $I^c_h = \{j_1,\dots,j_q\}$. Then there are $h_{j_a} \in G_f^d$ acting on $x_{j_a}$ only, s.t. $h = h_{j_1}\cdots h_{j_q}$. We have
\begin{align}
    \xi_{(\sigma,g)} \cup &\xi_{(\id,h)} = \xi_{(\sigma,g)}\cup \xi_{(\id,h_{j_1})} \cup \xi_{(\id,h_{j_2})} \cup \dots \cup \xi_{(\id,h_{j_q})}
    &  \textit{(by Proposition~\ref{prop: diagonal elements product})}
    \\
    &= (-1)^{k-1} \widetilde g_{j_1} \cdot \xi_{(\sigma,g h_{j_1})} \cup \xi_{(\id,h_{j_2})} \cup \dots \cup \xi_{(\id,h_{j_q})}.
    & \textit{(by Case 1 above)}
\end{align}
Now $(\sigma,g h_{j_1})$ is not special. By the same argument as in Case 1 above one gets immediately that ${\xi_{(\sigma,g h_{j_1})} \cup \xi_{(\id,h_{j_\bullet})}}$ can only be non-zero when $(\sigma,g h_{j_1} h_{j_\bullet})$ is special. This is equivalent to the fact that for $h_{j_1}$, we have started with, there is a pair group element $h_{j_\bullet}$, s.t. $h_{j_1}h_{j_\bullet} = 1$.

Assume $h=t_{i_a}^dt_{i_b}^{-d}$ with $i_a < i_b$. Then we have $\xi_{(\id,h)} = \xi_{t_{i_a}^d}\cup\xi_{t_{i_b}^{-d}}$. By Case 1 above we have two equalities
\begin{align}
    \xi_{(\sigma,g)}\cup\xi_{t_{i_a}^d} & =(-1)^{k-1}\widetilde{g}_{i_{a}}\cdot \xi_{(\sigma,g t^d_{i_a})}
    \\
    \xi_{(\sigma,g t^d_{i_a}t^{-d}_{i_b})}\cup\xi_{t_{i_b}^d} & =(-1)^{k-1}\widetilde{g}_{i_{b}} \zeta_n^{d} \cdot \xi_{(\sigma,gt^d_{i_a})}
\end{align}
implying together with Propositions~\ref{prop: diagonal elements product}
\begin{align}
\xi_{(\sigma,g)}\cup& \xi_{(\id,h)}  = \xi_{(\sigma,g)}\cup\xi_{t_{i_a}^d}\cup\xi_{t_{i_b}^{-d}}
= (-1)^{k-1} \widetilde{g}_{i_{a}} \cdot \xi_{(\sigma,gt^d_{i_a})}\cup \xi_{t_{i_b}^{-d}}
\\
& = \frac{\widetilde{g}_{i_{a}}}{\zeta^d \widetilde{g}_{i_{b}}}\cdot \xi_{(\sigma,gt^d_{i_a}t^{-d}_{i_b})}\cup\xi_{t_{i_b}^d}\cup \xi_{t_{i_b}^{-d}}
= 
\frac{\widetilde{g}_{i_{a}}}{\zeta^d \widetilde{g}_{i_{b}}}\cdot \xi_{(\sigma,gt^d_{i_a}t^{-d}_{i_b})} \cup \frac{n}{\zeta^d - 1} \lfloor x_{i_b}^{n-2} \rfloor \xi_\id 
\\
& = 
\frac{\widetilde{g}_{i_{a}}}{\zeta^d \widetilde{g}_{i_{b}}} \cdot \frac{n \left( \widetilde{g}_{i_{b}} \zeta^{d} \right)^{2}}{\zeta^d - 1} \cdot \lfloor \widetilde x_{i_1}^{n-2} \rfloor \xi_{(\sigma,gh)}
= 
\widetilde{g}_{i_{a}}\widetilde{g}_{i_{b}} \cdot n \frac{\zeta^d}{\zeta^d - 1} \cdot \lfloor \widetilde x_{i_1}^{n-2} \rfloor \xi_{(\sigma,gh)}.
\end{align} 
This gives the formula for $\xi_{(\sigma,g)}\cup \xi_{(\id,h)}$. The product $\xi_{(\id,h)} \cup \xi_{(\sigma,g)}$ is computed via the same technique.
\end{proof}

\begin{corollary}\label{corollary: vanishing}
    Let $\alpha,\beta \in G$ have a disjoint cycle decomposition $\alpha~=~\prod_{a=0}^p u_a$, $\beta~=~\prod_{b=0}^q v_b$ respectively. For any fixed $k \in \{1,\dots,N\}$ let $a$ and $b$ be s.t. $k \in I_{u_a}^c \cap I_{v_b}^c$.
    
    Then if neither $u_a$ nor $v_b$ is special and if $\det(u_a)\det(v_b)\neq \id$, we have $\xi_\alpha \cup \xi_\beta = 0$.
\end{corollary}
\begin{proof}
    We have $\xi_{\alpha} = \cup \xi_{u_\bullet}$ and $\xi_{\beta} = \cup \xi_{v_\bullet}$. The domains of $u_i$ and $u_j$ do not intersect for different indices $i,j$ and therefore we have $\xi_g \cup \xi_h = \pm \left(\cup_{i \neq a} \xi_{u_i} \right) \cup \xi_{u_a} \cup \xi_{v_b} \cup \left( \cup_{j \neq b} \xi_{v_j} \right)$. Take $G_f^d$--elements $\alpha' := t_k^r$ and $\beta' := t_k^s$ with $r$ and $s$, s.t. $\zeta^r_n = \det(u_a)$ and $\zeta^s_n = \det(v_b)$. By the proposition above we have $\xi_{u_a} \cup \xi_{v_b} = c \cdot \xi_{u_a'} \cup \xi_{t_k^r} \cup \xi_{t_k^s} \cup \xi_{v_b'}$ for $u_a',v_b'$ begin s.t. $u_a = u_a'\cdot t_k^r$, $v_b = t_k^s \cdot v_b'$ and non--zero constant $c$.
    
    By Proposition~\ref{prop: diagonal elements product} the product $\xi_{t_k^r} \cup \xi_{t_k^s}$ vanishes exactly when none of $u_a$, $v_b$ is special and $\det(u_a)\det(v_b)\neq \id$.
\end{proof}





\begin{corollary}\label{corollary: non-special decomposition}
    For any non--special $(\sigma,g) \in G$, $\sigma \neq \id$ there is special $(\sigma,g')$ and $(\id,g'')$, s.t. for some non--zero constant $c$ we have 
    \[
        \xi_{(\id,g'')} \cup \xi_{(\sigma,g')} = c \cdot \xi_{(\sigma,g)}.
    \]
\end{corollary}
\begin{proof}
    Let $\sigma$ be a cycle, denote $\psi:= \prod_{i \in I_\sigma^c} g_i$. We have $\psi = \zeta_n^d$ for some $d \not\in n\ZZ$. For any fixed $a \in I_\sigma^c$ set $g' := g t_a^{-d}$ and $g'' = t_a^d$. Corollary follows by the proposition above.
    
    For a general $\sigma$ consider the cycle decomposition of it and the respective decomposition of $g$. Perform the one cycle procedure above for every cycle in the decomposition. By using Eq~\eqref{eq: G--action on diagonalizable} the proof follows.
\end{proof}

The product $\xi_u \cup \xi_v$ of a non--special $u,v \in G \backslash G^s$ can be rather non--trivial. Assume $u = (i,j)t_i^pt_j^q$ with $i < j$ and $v = t_i^{r}$. By corollary above we know that $\xi_u \cup \xi_v = 0$ unless $r = -p-q$. Assume this to hold true. We have
\begin{align}
    \xi_{(i,j) t_i^p t_j^q} \cup \xi_{t_i^{-p-q}} 
    &= - \zeta^{-p} \cdot \xi_{(i,j) t_i^{p} t_j^{-p}} \cup \xi_{t_j^{p+q}} \cup \xi_{t_i^{-p-q}} && \textit{(by part 1 of Proposition ~\ref{prop: product of special with diagonal})}
    \\
    &= \zeta^{-p} \cdot \frac{n \zeta^{-q}}{1 - \zeta^{p+q}} \lfloor \widetilde x_i^{n-2}\rfloor \xi_{(i,j) t_i^{-q} t_j^{q}} && \textit{(by part 2 of Proposition ~\ref{prop: product of special with diagonal})}.
\end{align}
Note that we made the certain choice on the first step --- which group element $t_i^\bullet$ of $t_j^\bullet$ to split off. Following the other choice compared to one we did above we should make use of Proposition~\ref{prop: diagonal elements product} on the second step. That would basically reproduce the proof of part 2 of Proposition~\ref{prop: product of special with diagonal}.

\begin{proposition}\label{proposition: u and u^-1 nonspecial}
    For any non--special cycle $u = (\sigma,g)$ with $\sigma = (i_1,\dots,i_k)$ we have
    \[
        \xi_u \cup \xi_{u^{-1}} = n^k \frac{\det(g)}{1 - \det(g)} \left\lfloor x_{i_1}^{n-2} \cdots x_{i_k}^{n-2} \right\rfloor \xi_\id.
    \]
\end{proposition}
\begin{proof}
    Let the integer $a$ be s.t. $\det(g) = \zeta_n^{a}$. Then $g' := t_{i_k}^{-a}g$ is special. Denote $h := t_{i_1}^{-a}$. We have by Proposition~\ref{prop: product of special with diagonal}
    \begin{align}
        \xi_h \cup \xi_{(\sigma,g')} = \xi_{(\sigma,g)}, \quad \xi_{(\sigma,g')^{-1}} \cup \xi_{h^{-1}} = (-1)^{k-1} \xi_{(\sigma,g)^{-1}}.
    \end{align}
    Combining with together with Proposition~\ref{prop: cup-product of generalized cycles with inverse} we have
    \begin{align}
        \xi_{(\sigma,g)} \cup \xi_{(\sigma,g)^{-1}} &= (-1)^{k-1} \xi_h \cup \xi_{(\sigma,g')} \cup \xi_{(\sigma,g')^{-1}} \cup \xi_{h^{-1}}
        \\
        &=  n^{k-1} \xi_h \cup \lfloor \Phi^{(D)}_{i_1,\dots,i_k}(\bx)\rfloor \xi_\id \cup \xi_{h^{-1}}
        \\
        &=  n^{k-1} \left\lfloor \prod_{p=2}^k x_{i_p}^{n-2} \right\rfloor \xi_h \cup \xi_{h^{-1}}.
    \end{align}
    Computing the last cup--product via Proposition~\ref{prop: diagonal elements product} completes the proof.

\end{proof}

The following corollary approves our strategy.

\begin{corollary}\label{corollary: generators of HH}
The algebra $\A_{f,G}'$ is generated over $\CC[\bx]$ by $\xi_{(\sigma,g)}$ with $(\sigma,g)$ of the form $(\id,t_i^d)$ or $(i,j)\cdot t_i^{d}t_j^{-d}$.
\end{corollary}
\begin{proof}
By Lemma~\ref{prop: product of non-intersecting elements} it is sufficient to generate $\xi_{(\sigma,g)}$ for $\sigma$ being a single cycle $(i_1\ldots i_k)$ and $g$ acting by $1$ outside indices $\{i_1\ldots i_k\}$ . By Corollary~\ref{corollary: non-special decomposition} it is further sufficient to generate such elements with $g\in G^d$ -- special. Finally by Proposition~\ref{prop: special cycles intersecting by one index} for any such $g$ the we can decompose $\xi_{(\sigma,g)}$ into the product of vectors $\xi_{(i,j)t_{i}^{d_1}t_{j}^{-d_1}}$ for different exponents $d_\bullet$ and indices $i,j$. The statement follows.
\end{proof}

%

Recall the bigrading of $\A'_{f,G}$ introduced in Section~\ref{section: bigrading}. We have

\begin{proposition}\label{proposition: cup product preserves bigrading}
The $\cup$-product is compatible with the bigrading.
\end{proposition}

\begin{proof}
The cases of two charges are completely analogous. We will consider only the left one.

Since the degree of homogenous polynomials is additive with respect to product it is sufficient to consider products $\xi_{(\sig,g)}\cup\xi_{(\sig',g')}$. Note that that for elements acting nontrivially on nonintersecting sets of indices both ${\rm\age}{((\sig,g))}$ and $d_{(\sig,g)}$ are additive. It is sufficient to check this in the situations of  Proposition~\ref{prop: diagonal elements product}, Proposition~\ref{prop: special cycles intersecting by one index}, Proposition~\ref{prop: cup-product of generalized cycles with inverse} and Proposition~\ref{prop: product of special with diagonal}.

In the case of Proposition~\ref{prop: diagonal elements product} it is further sufficient to consider the case of the product $\xi_{t_k^p}\cup \xi_{t_k^{-p}}=\frac{n}{\zeta^p-1}\cdot\lfloor x_k^{n-2}\rfloor\xi_\id$, where we have $$-\frac{d_{t_k^p}}{n}+{\rm\age}{(t_k^p)}-\frac{d_{t_k^{-p}}}{n}+{\rm\age}{(t_k^{-p})}=-\frac{2}{n}+1=-\frac{\deg(x_k^{n-2})}{n}.$$

In the case of Proposition~\ref{prop: special cycles intersecting by one index} up to the multiplicative constant the product $\xi_u\cup\xi_v$ equals to $\xi_{u\cdot v}$ and we have 
\begin{align*}
-\frac{d_{u}}{n}&+{\rm\age}{(u)}-\frac{d_{v}}{n}+{\rm\age}{(v)}=-\frac{k-1}{n}+\frac{k-1}{2}-\frac{l-1}{n}+\frac{l-1}{2}
\\
&=-\frac{k+l-2}{n}+\frac{k+l-2}{2}=\frac{d_{uv}}{n}+{\rm\age}{(uv)}.
\end{align*}

In the case of Proposition~\ref{prop: cup-product of generalized cycles with inverse} we can further reduce ourselves to the situation of Lemma~\ref{square of mixed transposition}, i.e. $\xi_{(i,j)\cdot t_i^{d}t_j^{-d}} \cup\xi_{(i,j)\cdot t_i^{d}t_j^{-d}}= \left\lfloor\Phi_{ij}^{(d)}\right\rfloor\xi_{\id}$, where we have 
$$
2\left(-\frac{d_{(i,j)\cdot t_i^{d}t_j^{-d}}}{n}+{\rm\age}((i,j)\cdot t_i^{d}t_j^{-d})\right) = 2\left(-\frac{1}{n}+\frac{1}{2}\right) = 1-\frac{2}{n}=\frac{\deg\Phi_{ij}^{(d)}}{n}.
$$

In the case of Proposition~\ref{prop: product of special with diagonal} it is sufficient to consider the first option, where 
up to multiplicative constant the product $\xi_{(\sig,g)}\cup \xi_{t_i^d}$ equals $\xi_{(\sig,g)t_i^d}$. From Lemma~\ref{lemma: fixed set} and its proof we see that $d_{(\sig,g)t_i^d}=d_{(\sig,g)}+1$ and ${\rm\age}{((\sig,g)t_i^d)}={\rm\age}{((\sig,g))}+\frac{d}{n}$. This implies the statement.


\end{proof}


The following remark illustrates the independance of the $\cup$--product on the choice of the classes of polynomials.

\begin{remark}\label{remark: product well-defined}
In what follows denote by $\lfloor p \rfloor_u$ the class of $p$ in $\Jac(f^u)$.
Consider the product $\lfloor \phi \rfloor_u \xi_u \cup \lfloor \psi \rfloor_v \xi_v$ for some polynomials $\phi$ and $\psi$ supported on the fixed loci of $u$ and $v$ respectively. In the following two cases we show that the class $\lfloor \phi \psi \rfloor_{uv}$ does not depend on the choice of $\lfloor \phi \rfloor_u$ and $\lfloor \psi \rfloor_v$ representatives.

\subsubsection*{Case 1}
Let $u = (i_1,\dots,i_{p}) g$ and $v = (j_1,\dots,j_q) h$ be both special with $i_a = j_b$ for some indices $a,b$. Then $uv$ is a special cylce again and according to Proposition~\ref{prop: special cycles intersecting by one index}, $\sigma_{u,v}$ is a constant.
Assume also $i_1 \le j_1$. 

For some $r_\bullet \in \CC$ and $s_\bullet \in \CC[\bx]$ we have
\begin{align}
    \phi &= r_1(\frac{1}{p} \sum_{\alpha=1}^p \widetilde g_{i_\alpha} x_{i_\alpha})^{d_1} + s_1(\bx)(\sum_{\alpha=1}^{p} \widetilde g_{i_\alpha} x_{i_\alpha})^{n-1},
    \\
    \psi &= r_2(\frac{1}{q} \sum_{\beta=1}^q \widetilde h_{j_\beta} x_{j_\beta})^{d_2} + s_2(\bx)(\sum_{b=1}^q \widetilde h_{j_\beta} x_{j_\beta})^{n-1}.
\end{align}
In $\Jac(f^{uv})$ apply Eq.\eqref{formula: coordinates on sector} and \eqref{formula: coordinate vanishing on sector}.
\begin{align}
    \phi \cdot \psi = r_1 r_2 \widetilde g_{i_a} \cdot \widetilde x_{uv}^{d_1} \widetilde x_{uv}^{d_2} + \gamma(\bx)\cdot \widetilde x_{uv}^{n-1}
\end{align}
for some $\gamma(\bx)$. The class $\lfloor \phi \cdot \psi \rfloor_{uv}$ obviously does not depend on $s_1(\bx),s_2(\bx)$.
Having applied Eq.\eqref{formula: coordinates on sector} and \eqref{formula: coordinate vanishing on sector} before taking the product, we should note that we have $\widetilde x_u = \widetilde x_{uv}$ and $\widetilde x_{uv} = \widetilde g_{i_{a}} \widetilde x_v$, what gives the statement.

\subsubsection*{Case 2}
Let $u = (i_1,\dots,i_{p}) g$ be special and $v = u^{-1}$. For $\phi,\psi$ expressed as above consider $\lfloor \phi \rfloor_u \xi_u \cup \lfloor \psi \rfloor_{u^{-1}} \xi_{u^{-1}}$.
We have by Proposition~\ref{prop: cup-product of generalized cycles with inverse} that $\sigma_{u,u^{-1}} = \Phi_{i_1,\dots,i_k}^{(D)}(\bx)$. In order for the class of the product to be well-defined, we need $\lfloor \phi \psi \Phi^{(D)}_{i_1,\dots,i_k} \rfloor_\id$ not to depend on $s_1(\bx)$ and $s_2(\bx)$. 

This follows immediately by degree counting. Namely, the total degree of the polynomial $\Phi^{(D)}_{i_1,\dots,i_k} \widetilde x_u^{n-1}$ is $(n-2)k + 1$. Therefore it vanishes in $\Jac(f)$.
\end{remark}

\begin{theorem}\label{theorem: Frobenius algebra before invariants}
The pairing $\eta_{f,G}$ is Frobenius in a sense that  
\[  
\eta_{f,G}(\lfloor\phi'\rfloor\xi_{u_1}\cup  \lfloor\phi''\rfloor\xi_{u_2}, \lfloor\phi'''\rfloor\xi_{u_3})=\eta_{f,G}(\lfloor\phi'\rfloor\xi_{u_1},  \lfloor\phi''\rfloor\xi_{u_2}\cup\lfloor\phi'''\rfloor\xi_{u_3}). \]
    
\end{theorem}
\begin{proof}

It suffices to check that the value of the bilinear form depends only on the product of the arguments. In non tautologically vanishing cases this means checking  \[\eta_u \left(\lfloor \phi' \rfloor \xi_u, \lfloor \phi'' \rfloor \xi_{u^{-1}} \right)=\eta_\id(\xi_\id,\lfloor \phi' \rfloor\xi_u\cup \lfloor \phi'' \rfloor\xi_{u^{-1}}).\]

By definition, the left hand side is equal to $\eta_{f^u}(\lfloor \phi' \rfloor, \lfloor \phi'' \rfloor)=\eta_{f^u}(\lfloor 1 \rfloor, \lfloor \phi'\phi'' \rfloor)$. Note that we may assume $\phi''$ to be $u^{-1}$ and, thus, $u$ invariant. This allows us to rewrite the right hand side as  $\eta_\id(\lfloor1\rfloor,\lfloor \phi' \phi''\rfloor\xi_u\cup \xi_{u^{-1}})$. Therefore, we have to check that  
\[
    \eta_\id(\xi_\id, [\phi' \phi'']\xi_u\cup \xi_{u^{-1}}) =
    \begin{cases}
     c (n-1)^{N_u}  \quad & \text{ if } \exists c\in \CC, \text{ s.t. } [\phi' \phi''] = c \lfloor \mathcal{H}_u\rfloor
     \\
     0 & \text{ otherwise.}
    \end{cases}
\]
We consider some special cases and derive then the general situation.

{\bf Case 1.} Let $u=(\sig, g)$ be a length $k$ non--special cycle. By Proposition~\ref{proposition: u and u^-1 nonspecial} we have 
\begin{align*}
\lfloor & \mathcal{H}_u \rfloor\xi_u\cup \xi_{u^{-1}}=(\det(g)^{-1}-1) \left\lfloor \prod_{a \not\in I_{u}^c} n(n-1) x_{a}^{n-2} \right\rfloor \cdot n^k \frac{\det(g)}{1 - \det(g)} \left\lfloor \prod_{a \in I_{u}^c} x_{a}^{n-2} \right\rfloor \xi_\id 
\\
&= (n-1)^{-k}\lfloor \mathcal{H}_\id\rfloor\xi_\id, 
\end{align*}
and the statement follows by 
\[
 \eta_\id(\xi_\id, (n-1)^{-k}\lfloor \mathcal{H}_\id\rfloor\xi_\id) = (n-1)^{-k} \eta_\id(\xi_\id, \lfloor \mathcal{H}_\id\rfloor\xi_\id) = (n-1)^{N-k}.
\]

Other monomials in $\Jac(f_u)$ are of smaller degree and, thus, $\left\lfloor \phi\prod_{a \not\in I_{u}^c} n(n-1) x_{a}^{n-2} \right\rfloor$ is of degree smaller than degree of $\mathcal{H}_\id$, hence $\eta_\id(\xi_\id,\lfloor \phi\rfloor\xi_u\cup \xi_{u^{-1}})=0$ for such monomial.

{\bf Case 2.} Let $u$ be a length $k$ special cycle. By Proposition~\ref{prop: cup-product of generalized cycles with inverse} we have
\begin{align*}
\lfloor &\mathcal{H}_u \rfloor\xi_u\cup \xi_{u^{-1}} = \left\lfloor (-1)^{k-1}n(n-1)\tilde{x}_{i_1}^{n-2} \prod_{a \not\in I_{u}^c} n(n-1) x_{a}^{n-2} \right\rfloor \cdot (-n)^{k-1} \left\lfloor \Phi^{(D)}_{i_1,\dots,i_k}(\bx) \right\rfloor \xi_\id
\\
&=n^{k}(n-1)\left\lfloor \prod_{a \not\in I_{u}^c} n(n-1) x_{a}^{n-2} \cdot \tilde{x}_{i_1}^{n-2}\cdot \Phi^{(D)}_{i_1,\dots,i_k}(\bx) \right\rfloor 
\\
& = \textit{(Eq.~\eqref{formula: coordinates on sector} and \eqref{formula: coordinate vanishing on sector})}
\\
& = n^{k}(n-1)\left\lfloor \prod_{a \not\in I_{u}^c} n(n-1) x_{a}^{n-2} \cdot x_{i_1}^{n-2}\cdot x_{i_2}^{n-2}x_{i_3}^{n-2}\ldots x_{i_k}^{n-2}\right\rfloor
\\
& = (n-1)^{-k+1}\lfloor \mathcal{H}_\id\rfloor\xi_\id.
\end{align*}
Other monomials in $\Jac(f_u)$ are of smaller degree and, thus, $\left\lfloor \phi\Phi^{(D)}_{i_1,\dots,i_k}(\bx) \right\rfloor$ is of degree smaller than degree of $\mathcal{H}_\id$, hence $\eta_\id(\xi_\id,\lfloor \phi\rfloor\xi_u\cup \xi_{u^{-1}})=0$ for such monomial.

{\bf Case 3.} Let $u = (\id,g)$. This case follows from \cite{BT2} combined with \cite{BTW16} but we repeat the proof here for completeness.
We have by using main property of $H_{g,h}$
\begin{align}
    \eta_{f,G} & \left( \xi_\id, \frac{\hess(f^g)}{\mu_{f^g}} \xi_g \cup \xi_{g^{-1}} \right)  =  \eta_{f,G} \left( \xi_\id, (-1)^{\frac{d_g(d_g-1)}{2}} \prod_{i \in I_g^c}\frac{1}{g_i-1}\frac{\hess(f^g)}{\mu_{f^g}} H_{g,g^{-1}} \xi_\id \right)
    \\
    & = (-1)^{\frac{d_g(d_g-1)}{2}} \prod_{i \in I_g^c} \frac{1}{g_i-1} \eta_{g} \left( [1], \frac{[\hess(f^g)]}{\mu_{f^g}} \right)
    = \eta_{f,G} \left( \frac{\hess(f^g)}{\mu_{f^g}} \xi_g, \xi_{g^{-1}} \right).
\end{align}
From where we conclude that $\eta([\phi]\xi_g \cup [\psi]\xi_h, \xi_{(gh)^{-1}}) = \eta(\xi_\id, [\phi]\xi_g \cup [\psi]\xi_h \cup \xi_{(gh)^{-1}})$ what completes the proof of this case.

{\bf Case 4.} Consider arbitray $u\in G$. Let it be decomposed into a product of non--intersecting cycles $u = (\id,g_0) \prod_{i=1}^k (\sigma_i,g_i)$. Setting $\sigma_0 := \id$ we have
\[
    \lfloor \mathcal{H}_u \rfloor \xi_u \cup \xi_{u^{-1}} = \bigcup_{i=0}^k \lfloor \mathcal{H}_{(\sigma_i,g_i)} \rfloor \xi_{(\sigma_i,g_i)} \cup \xi_{(\sigma_i,g_i)^{-1}}.
\]
By using the computations of the cases above the statement follows.
\end{proof}

\section{The $G$-action on $\A'_{f,G}$}\label{section: G--action}
The purpose of this section is to compute the Hochschild cohomology ring $\ccHH^*(\CC[\bx]\rtimes G,f)$. According to Theorem~\ref{theorem: Shklyarov}, it's isomorphic to $\left( \A'_{f,G}\right)^G$ and we need to compute the $G$--invariant subspace of $\A'_{f,G}$.
Note that in \cite{S20} this group action is not calculated for the noncommutative groups (see Remark 4.18 loc.cit.). 

\subsection{The formulae for the $G$--action}\label{section: G--action formulae}
By Eq.~\eqref{eq: G--action on diagonalizable} we have for $d\not\equiv 0 \mod n$
\begin{equation}\label{diagonal action}
t_i^*(\xi_{t_j^d})=
 \begin{cases}
      \xi_{t_j^d} \quad & \text{ if } i\ne j,
        \\
    \zeta_n^{-1}\xi_{t_j^d} \quad & \text{ if } i=j.        
\end{cases}
\end{equation}

Let us now determine the action of the diagonal subgroup on $\A'_{f,G}$ by computing its action on the elements $\xi_{(i,j)\cdot t_i^{d}t_j^{-d}}$

\begin{proposition}\label{action of diagonal}
For any $i < j$, $k \neq i,j$ and $d \in \ZZ$ we have
\begin{description}
 \item[(1)] $t_k^*(\xi_{(i,j)\cdot t_i^{d}t_j^{-d}})=\xi_{(i,j)\cdot t_i^{d}t_j^{-d}}$,
 \item[(2)] $t_i^*(\xi_{(i,j)\cdot t_i^{d}t_j^{-d}})=\xi_{(i,j)\cdot t_i^{d-1}t_j^{-d+1}}$,
 \item[(3)] $t_j^*(\xi_{(i,j)\cdot t_i^{d}t_j^{-d}})=\zeta_n^{-1}\xi_{(i,j)\cdot t_i^{d+1}t_j^{-d-1}}$.
\end{description}

\end{proposition}

We will need the following 
\begin{lemma}\label{injection}
The map $ \bullet \cup\xi_{(i,j)\cdot t_i^{d}t_j^{-d}}: \A'_{(i,j)\cdot t_i^{d}t_j^{-d}} \rightarrow \A'_\id$ is an injection.
\end{lemma}
\begin{proof}
By Proposition~\ref{prop: cup-product of generalized cycles with inverse} the image of the map is the ideal $I\subset\A'_\id$ generated by $\lfloor \Phi^{(d)}_{ij}\rfloor$. Since in $\A'_\id/I$ we have $x_j^{n-2}=-\sum^{n-2}_{a=1} \zeta_n^{-da}x_i^ax_j^{n-2-a}$ the quotient ring has a basis $\{x_1^{e_1}\ldots x_N^{e_N}\}$ with $e_k<n-1$ for $k\ne j$ and $e_j<n-2$. Therefore, 
\begin{align*}
 & \dim\A'_\id/I=(n-1)^{N-1}(n-2),
 \\
 & \dim I=\dim\A'_\id-\dim\A'_\id/I=(n-1)^N-(n-1)^{N-1}(n-2)=(n-1)^{N-1}.
\end{align*}
Since we also have $\dim\A'_{(t_i^{d}t_j^{-d},(i,j))}=(n-1)^{N-1}$ the statement follows.
\end{proof}

\begin{proof}[Proof of Proposition~\ref{action of diagonal}]
{\bf (1).} 
Since $t_k$ and $(i,j) t_i^{d}t_j^{-d}$ commute as $G$--elements, there is a polynomial $P \in \CC[\bx]$ preserved by $(i,j)t_i^{d}t_j^{-d}$, s.t. $t_k^*(\xi_{(i,j)\cdot t_i^{d}t_j^{-d}}=\lfloor P\rfloor\xi_{(i,j)\cdot t_i^{d}t_j^{-d}}$. We then have by using Propositon~\ref{prop: cup-product of generalized cycles with inverse} and braided commutativity
\begin{align*}
    \lfloor \Phi^{(d)}_{ij}\rfloor\xi_\id  &= t_k^*\left(\lfloor \Phi_{ij}^{(d)}\rfloor\xi_\id\right) = t_k^*\left(- n^{-1} \xi_{(i,j)\cdot t_i^{d}t_j^{-d}}\cup\xi_{(i,j)\cdot t_i^{d}t_j^{-d}}\right)
    \\
    &= - n^{-1} \cdot t_k^*(\xi_{(i,j)\cdot t_i^{d}t_j^{-d}})\cup t_k^*(\xi_{(i,j)\cdot t_i^{d}t_j^{-d}})
    =- n^{-1} \cdot \lfloor P\rfloor\xi_{(i,j)\cdot t_i^{d}t_j^{-d}}\cup\lfloor P\rfloor\xi_{(i,j)\cdot t_i^{d}t_j^{-d}}
    \\
    &= - n^{-1} \cdot  \lfloor P^2\rfloor\xi_{(i,j)\cdot t_i^{d}t_j^{-d}}\cup\xi_{(i,j)\cdot t_i^{d}t_j^{-d}}= \lfloor P^2 \Phi^{(d)}_{ij}\rfloor\xi_\id
\end{align*}
We see from Lemma~\ref{injection} that $\lfloor \Phi^{(d)}_{ij}\rfloor=\lfloor P^2 \Phi^{(d)}_{ij}\rfloor$ in identity sector implies $\lfloor P^2\rfloor=\lfloor 1\rfloor$ in $\A_{(i,j)\cdot t_i^{d}t_j^{-d}}$. We then have $P(0)=\pm1$. If $P$ has a term of a smallest non zero degree less than $n-1$, then $P^2$ will also have the term of the same degree implying $\lfloor P\rfloor=\lfloor \pm1\rfloor$. Now to see that $\lfloor P\rfloor=\lfloor 1\rfloor$ we make us of the other products we know. Assume $k > i$. Braided commutativity and Proposition~\ref{prop: product of non-intersecting elements} imply
\begin{align*}
    \xi_{(i,j)\cdot t_k^{-1}t_i^{d}t_j^{-d}} &= \xi_{(i,j)\cdot t_i^{d}t_j^{-d}}\cup\xi_{t_k^{-1}}=-\xi_{t_k^{-1}}\cup t_k^*( \xi_{(i,j)\cdot t_i^{d}t_j^{-d}})
    \\
    &=-\xi_{t_k^{-1}}\cup\lfloor P\rfloor \xi_{(i,j)\cdot t_i^{d}t_j^{-d}}=\lfloor P\rfloor \xi_{(i,j)\cdot t_k^{-1}t_i^{d}t_j^{-d}}
  \end{align*}
and the statement follows. The same reasoning works for $k < i$ with the other sequence of signs.

{\bf (2) and (3).} First note that $t_i^*(\Phi^{(d)}_{ij}) = \Phi^{(d-1)}_{ij}$ and $t_j^*(\Phi^{(d)}_{ij})=\zeta_n^{-2} \Phi^{(d+1)}_{ij}$. The rest of the argument is completely parallel to (1) by employing Proposition~\ref{prop: product of special with diagonal} to compute products.

\end{proof}
\begin{corollary}\label{action of diag on special}
    Let $(\sigma,g)\in G$ be special with $\sigma = (i_1,\dots,i_k)$ and $h \in G^d$. We have
    \begin{equation}
        h^*\left(\xi_{(\sigma,g)}\right) = \left(\prod_{p=2}^k h_{i_p} \right)^{-1}\xi_{h(\sigma,g)h^{-1}}
    \end{equation}
\end{corollary}
\begin{proof}
    This is immediate by using the decomposition of $\xi_{(\sigma,g)}$ as in proof of Proposition~\ref{prop: cup-product of generalized cycles with inverse}.
\end{proof}
\begin{corollary}\label{action of diag on nonspecial}
    Let $(\sigma,g)\in G$ be non-special and $h \in G^d$. Denote $I := I^c_{(\sigma,g)}$. We have
    \begin{equation}
        h^*\left(\xi_{(\sigma,g)}\right) = \left(\prod_{i \in I} h_{i} \right)^{-1}\xi_{h(\sigma,g)h^{-1}}
    \end{equation}
\end{corollary}
\begin{proof}
    By Corollary~\ref{corollary: non-special decomposition} we have $h^*(\xi_{(\sigma,g)}) = h^*(\xi_{u} \cup \xi_{v})$ for $u,v \in G$, s.t. $v$ is special and $u \in G^d$ acting on the first indices of every cycle of $u$. By corollary above the action of $h$ of $\xi_{v}$ produces the product of those $h_\bullet^{-1}$ whose index is in $I$ and doesn't coincide with the first index of the any cycle. The action of $h$ on $\xi_{u}$ produces the product of those $h_\bullet^{-1}$ whose index is in $I$ and coincides with the first index of the any cycle. All together these give us the multiple written.
\end{proof}


Next we compute the action of the symmetric group on the diagonal sectors

\begin{proposition}\label{action of sym on diag}
For any $(\sigma,\id) \in G$ and $g \in G^d$ we have
\begin{align}
    & (\sigma,\id)^* ( \xi_{g}) = \xi_{(\sigma,\id)g(\sigma,\id)^{-1}}.
    \label{eq: Gact 1}
\end{align}
\end{proposition}
\begin{proof}
It's enough to show that $(i,j)^* ( \xi_{t_k^d}) = \xi_{(i,j)t_k^d(i,j)}$ for any indices $i,j,k$ and $d \in \ZZ$. 
Put $(i,j)^* ( \xi_{t_k^d}) =\lfloor P\rfloor \xi_{(i,j)t_k^d(i,j)}$ for some polynomial $P \in \CC[x_i,x_j]$ of degree at most $n-2$ in each variable. By Eq~\eqref{diagonal action} we see that for each $l\in\{1,\ldots, N\}$ both $(i,j)^* ( \xi_{t_k^d})$ and  $\xi_{(i,j)t_k^d(i,j)}$ are eigenvectors of $t_l^*$ with the same eigenvalue, which implies that $P$ is a constant. To establish $P=1$ we  once again use the braided commutativity. If $k\ne i,j$ elements $\xi_{t_k^d}$ and $\xi_{(i,j)}$ anticommute by Proposition~\ref{prop: product of non-intersecting elements} implying $P=1$. 

If $k=i$ we have by Proposition~\ref{prop: product of special with diagonal}
\begin{align*}
\xi_{(i,j)\cdot t_j^d}=& \xi_{t_i^d}\cup\xi_{(i,j)}=-\xi_{(i,j)}\cup(i,j)^*(\xi_{t_i^d})
\\
& =- \xi_{(i,j)}\cup \lfloor P\rfloor \xi_{t_j^d}=\lfloor P\rfloor\xi_{(i,j)\cdot t_j^d}
\end{align*}
implying $P=1$. The case $k=j$ is completely analogous.

\end{proof}

Finally we determine the action of $(i,j)$ on the sectors of ${(k,l)t_{k}^{d}t_{l}^{-d} \in G_f}$.

\begin{proposition}\label{trans on trans} Consider pairwise different indices $i,j,k,l$ any $d \in \ZZ$. We have
\begin{description}
 \item[(1)] $(i,j)^*(\xi_{(i,j) \cdot t_i^{d}t_j^{-d}})=-\zeta_n^{d} \cdot \xi_{(i,j)\cdot t_i^{-d}t_j^{d}}$.
 \item[(2)] $(i,j)^*(\xi_{(k,l)\cdot t_{k}^{d}t_{l}^{-d}})=\xi_{(k,l)\cdot t_{k}^{d}t_{l}^{-d}}$.
 \item[(3)] Assume $i<j<k$. We have 
    \begin{align*}
        & (i,k)^* (\xi_{(i,j) \cdot t_i^{d}t_j^{-d}}) = - \zeta_n^{d}\xi_{(j,k) \cdot t_j^{-d}t_k^{d}}, \quad (i,j)^* (\xi_{(i,k) \cdot t_i^{d}t_k^{-d}}) = \xi_{(j,k) \cdot t_j^{d}t_k^{-d}},
        \\
        & (i,k)^* (\xi_{(j,k) \cdot t_j^{d}t_k^{-d}}) = - \zeta_n^{d}\xi_{(i,j) \cdot t_i^{-d}t_j^{d}}, \quad (i,j)^* (\xi_{(j,k) \cdot t_j^{d}t_k^{-d}}) = \xi_{(i,k) \cdot t_i^{d}t_k^{-d}},
        \\
        & (j,k)^* (\xi_{(i,k) \cdot t_i^{d}t_k^{-d}}) = \xi_{(i,j) \cdot t_i^{d}t_j^{-d}}, \quad (j,k)^* (\xi_{(i,j) \cdot t_i^{d}t_j^{-d}}) =  \xi_{(i,k) \cdot t_i^{d}t_k^{-d}}.
    \end{align*}
\end{description}
\end{proposition}
\begin{proof}~
\\
{\bf (1).} We have  $((i,j) t_i^{d}t_j^{-d})^*(\xi_{(i,j)\cdot t_i^{d}t_j^{-d}})=\lfloor P\rfloor\xi_{(i,j)\cdot t_i^{d}t_j^{-d}}$ for some polynomial $P$ preserved by $(i,j)t_i^{d}t_j^{-d}$. Noting that $((i,j) t_i^{d}t_j^{-d})^*(\Phi^{(d)}_{ij}) = \Phi^{(d)}_{ij}$ we have by Proposition~\ref{prop: cup-product of cycles with inverse}
\begin{align*}
    \lfloor \Phi^{(d)}_{ij}\rfloor\xi_\id  &= ((i,j) t_i^{d}t_j^{-d})^*\left(\lfloor \Phi_{ij}^{(d)}\rfloor\xi_\id\right) = ((i,j) t_i^{d}t_j^{-d})^*\left( - n^{-1}\xi_{(i,j)\cdot t_i^{d}t_j^{-d}}\cup\xi_{(i,j)\cdot t_i^{d}t_j^{-d}}\right)=
    \\
    &= - n^{-1} \cdot ((i,j) t_i^{d}t_j^{-d})^*(\xi_{(i,j)\cdot t_i^{d}t_j^{-d}})\cup ((i,j) t_i^{d}t_j^{-d})^*(\xi_{(i,j)\cdot t_i^{d}t_j^{-d}})
    \\
    &= - n^{-1} \cdot  \lfloor P\rfloor\xi_{(i,j)\cdot t_i^{d}t_j^{-d}}\cup\lfloor P\rfloor\xi_{(i,j)\cdot t_i^{d}t_j^{-d}}
    \\
    &= - n^{-1} \cdot \lfloor P^2\rfloor\xi_{(i,j)\cdot t_i^{d}t_j^{-d}}\cup\xi_{(i,j)\cdot t_i^{d}t_j^{-d}}= \lfloor P^2 \Phi^{(d)}_{ij}\rfloor\xi_\id.
\end{align*}
Proceeding in the same way as in proof of Proposition~\ref{action of diagonal} we see that $\lfloor P\rfloor=\pm 1$. 
The following computation shows that $P=-1$.
\begin{align*}
   \lfloor \Phi^{(d)}_{ij}\rfloor\xi_\id &= - n^{-1} \cdot \xi_{(i,j)\cdot t_i^{d}t_j^{-d}}\cup\xi_{(i,j)\cdot t_i^{d}t_j^{-d}} = n^{-1} \cdot \xi_{(i,j)\cdot t_i^{d}t_j^{-d}}\cup ((i,j) t_i^{d}t_j^{-d})^*( \xi_{(i,j)\cdot t_i^{d}t_j^{-d}})
    \\
    &= n^{-1} \cdot  \xi_{(i,j)\cdot t_i^{d}t_j^{-d}}\cup\lfloor P\rfloor \xi_{(i,j)\cdot t_i^{d}t_j^{-d}} = - \lfloor P \Phi^{(d)}_{ij}\rfloor \xi_\id.
  \end{align*}
Now $(i,j)=t_i^{d}t_j^{-d}((i,j)\cdot t_i^{d}t_j^{-d})$ and the statement follows from Proposition~\ref{action of diagonal} parts (2) and (3) applied both multiple times.

\noindent
{\bf (2).} Since $((i,j) t_i^{d}t_j^{-d})^*(\Phi^{(d)}_{kl}) = \Phi^{(d)}_{kl}$  the same argument as above implies $\lfloor P\rfloor=\pm 1$. Since $\xi_{(i,j)\cdot t_i^{d}t_j^{-d}}$ and $\xi_{(k,l)\cdot t_{k}^{d}t_{l}^{-d}}$ anticommute by Proposition~\ref{prop: product of non-intersecting elements} we have $\lfloor P\rfloor=1$.

\noindent
{\bf (3).} We have 
\begin{align*}
    & (i,k)^* (\Phi^{(d)}_{ij}) = \zeta_n^{2d} \Phi^{(-d)}_{jk}, \quad (i,j)^* (\Phi^{(d)}_{ik}) = \Phi^{(d)}_{jk}, \quad (i,k)^* (\Phi^{(d)}_{jk}) = \zeta_n^{2d} \Phi^{(-d)}_{ij},
    \\
    & (i,j)^* (\Phi^{(d)}_{jk}) = \Phi^{(d)}_{ik}, \quad (j,k)^* (\Phi^{(d)}_{ik}) = \Phi^{(d)}_{ij}, \quad (j,k)^* (\Phi^{(d)}_{ij}) = \Phi^{(d)}_{ik}.
    \end{align*}
The same argument as above allows us to fix the action by computing products via Proposition~\ref{prop: special cycles intersecting by one index}.
\end{proof}
\begin{corollary}
    For any special $u = (i_1,\dots,i_k) g \in G$ with $g \in G^d$ we have 
    \[
    u^* (\xi_u) = (-1)^{k-1} \xi_u. 
    \]
\end{corollary}
\begin{proof}
    This follows immediately from the proposition above by decomposing the cycle into transpositions.
\end{proof}

\begin{corollary}
    Let $u,v \in G$ be special with $u = (i_1,\dots,i_k) g$, $v = (j_1,\dots,j_l) h$ for some $g, h \in G^d$ and $i_a = j_b$ for some $a,b$. 
    Then we have
    \[
        (v^{-1})^* \left( \xi_u \right) = 
        \begin{cases}
            \dfrac{\eps_{u,v}}{\eps_{v,v^{-1}uv}} h_{j_b} \xi_{v^{-1}\cdot u \cdot v} \quad \text{ if } i_1 < j_1,
            \\
            \dfrac{\eps_{u,v}}{\eps_{v,v^{-1}uv}} g_{i_a} \xi_{v^{-1}\cdot u \cdot v} \quad \text{ if } i_1 \ge j_1
        \end{cases}        
    \]
    for $\eps_{u,v}$ as in Proposition~\ref{prop: Clifford multiplications}. 
\end{corollary}

\begin{example}
    Let ${G = \lbrace \id, (i,j)(k,l), (i,k)(j,l), (i,l)(j,k) \rbrace}$ for some indices ${i<j<k<l}$. We have
    \begin{align}
        & \left( (i,j)(k,l) \right)^* \xi_{(i,k)(j,l)} = (i,j)^* (k,l)^* \left( \xi_{(i,k)} \cup \xi_{(j,l)} \right)
        \\
        & \quad = (i,j)^* \left( \xi_{(i,l)} \cup \xi_{(j,k)} \right) = \xi_{(j,l)} \cup \xi_{(i,k)} = - \xi_{(i,k)} \cup \xi_{(j,l)} = - \xi_{(i,k)(j,l)}.
    \end{align}
    Similarly we have
    \begin{align}
        & \left((i,l)(j,k) \right)^* \xi_{(i,k)(j,l)} = -\xi_{(i,k)(j,l)}.
    \end{align}

\end{example}

\subsection{Invariants of the group action}

Fix a subgroup $G \subseteq G_f$. 
For any $u \in G$ let $Z(u)\subseteq G$ be the centralizer of $u$ and  $C(u)\simeq G/Z(u)$ be the conjugacy class of $u$ in $G$. Note that $Z(u) \cdot \A'_{u} \subseteq \A'_{u}$ and $G\cdot \A'_{u}\subseteq \bigoplus_{v \in C(u)}\A'_{v}$.

\begin{proposition}
The symmetrization map 
\begin{align}
    \mathrm{sym}\colon \ \A'_{u} &\to \Big(\bigoplus_{v \in C(u)}\A'_{v}\Big)^G \subseteq \left( \A'_{f,G} \right)^G
    \\
    x &\mapsto \sum_{w\in G} w^*(x).
\end{align}
induces an isomorphism 
\[
\Big( \A'_{u} \Big)^{Z(u)} \xrightarrow[]{\sim} \Big(\bigoplus_{v \in C(u)} \A'_{v} \Big)^G.
\]
\end{proposition}

\begin{proof}
Let $\mathrm{pr}: \A'_{f,G} \to \A'_{f,G}$ be the projection on the component $\A'_{u} \subset \A'_{f,G}$. Consider the composition $\mathrm{pr}\circ\mathrm{sym}$ restricted to $(\A'_{u})^{Z(u)} \subseteq \A'_u$. 

Restricted to $(\A'_{u})^{Z(u)}$, the map $\mathrm{sym}$ acts by multiplication by the order of $Z(u)$. 


Since $G$ acts transitively on $C(u)$, an element of $\big(\bigoplus_{v \in C(u)}\A'_{v}\big)^G$ is uniquely defined by its image under $\mathrm{pr}$, which is, therefore, an injection. It follows, that $\mathrm{pr}$ and $\mathrm{sym}$ are mutually inverse up to multiplication by the order of $Z(u)$.
\end{proof}


\begin{corollary}
Let $\C^G\subset G$ be a set of representatives of conjugacy classes in $G$. 
There is an isomorphism of bigraded vector spaces 
\[
(\A'_{f,G})^G\simeq \bigoplus_{u\in \C^G} \left( \A'_{u} \right)^{Z(u)},
\]
where the action of $Z(u)$ on the right hand side is linear (see Section~\ref{section: G on HH from def}).
\end{corollary}
\begin{proof}
    It follows immediately from the proposition above that $\mathrm{sym}$ establishes an isomorphism 
    \[
    (\A'_{f,G})^G = \bigoplus_{u \in \C^G} \Big( \bigoplus_{v\in C(u)} \A'_v \Big)^G \simeq \bigoplus_{u\in \C^G} \left(\A'_{u} \right)^{Z(u)}.
    \]
    Fix some $u \in \C^G$. Recall that $\p_{\widetilde \theta_{u}}$ is constructed in the basis, diagonalizing the action of $u$. 
    The same basis diagonalizes the action of any $v \in Z(u)$ becase $u$ and $v$ commute. Due to this reason the Hochschild cohomology group action can be computed by Eq.~\eqref{eq: G--action on diagonalizable}.
    
    In order to show that $\mathrm{sym}$ preserves the bigrading $(q_l,q_r)$ it's enough to show that the $G$--action preserves this grading. This follows directly from the computations we have done in Section~\ref{section: G--action formulae}.
\end{proof}


\begin{example}
    Let $N = 4$ and $G = A_4\subset S_4$. Consider an element $u=(1,2)(3,4)$. Its conjugacy class consists of elements $(1,2)(3,4), (1,3)(2,4)$ and $(1,4)(2,3)$ and its centraliser is generated by $(1,2)(3,4)$ and $(1,3)(2,4)$.
    
 By  Proposition~\ref{trans on trans} we have 
\begin{align*}
& ((1,2)(3,4))^*(\xi_{(1,2)(3,4)})=\xi_{(1,2)(3,4)}, 
\\
& ((1,3)(2,4))^*(\xi_{(1,2)(3,4)})=-\xi_{(1,2)(3,4)}.
\end{align*}   
Thus, the invariants $(\A'_{(1,2)(3,4)})^{Z((1,2)(3,4))}$ consists of elements $\lfloor \phi(\widetilde x_1, \widetilde x_3)\rfloor\xi_{(1,2)(3,4)}$ such that $\phi(\widetilde x_1, \widetilde x_3)=-\phi(\widetilde x_3, \widetilde x_1)$.

We further have $(1,2,3)^*\xi_{(1,2)(3,4)}=-\xi_{(1,4)(2,3)}$, $(1,3,2)^*\xi_{(1,2)(3,4)}=-\xi_{(1,3)(2,4)}$. We conclude that in this case \[\mathrm{sym}\colon  (\A'_{(1,2)(3,4)})^{Z((1,2)(3,4))} \to \Big(\A'_{(1,2)(3,4)}\oplus \A'_{(1,3)(2,4)}\oplus \A'_{(1,4)(2,3)}\Big)^{A_4}\] is given by 
\[
\lfloor \phi(\widetilde x_1, \widetilde x_3) \rfloor\xi_{(1,2)(3,4)} \mapsto 4\lfloor \phi(\widetilde x_1, \widetilde x_3) \rfloor\xi_{(1,2)(3,4)}-4\lfloor \phi(\widetilde x_1, \widetilde x_2) \rfloor\xi_{(1,3)(2,4)}-4\lfloor \phi(\widetilde x_2, \widetilde x_1) \rfloor\xi_{(1,4)(2,3)}.
\]

\end{example}

We now prove that if $G^d\subseteq \SL_N(\CC)$ the pairing $\eta_{f,G}$ restricted to $(\A_{f,G}')^G$ is non--degenerate. 
We first prove the following.

\begin{proposition}\label{proposition: form G-invariant} Let $G^d \subseteq\SL_N(\CC)$. Then the pairing $\eta_{f,G}$ is $G$-invariant.
\end{proposition}
\begin{proof}

For $u, v\in G$ by Theorem~\ref{theorem: Frobenius algebra before invariants} and braided commutativity we have 
\begin{align*}
& \eta_{f,G}(v^*(\lfloor\phi\rfloor\xi_u), v^*(\lfloor\phi'\rfloor\xi_{u^{-1}}))= \eta_{f,G}(\xi_\id, v^*(\lfloor\phi\rfloor\xi_u)\cup v^*(\lfloor\phi'\rfloor\xi_{u^{-1}}))
\\
&\quad =\eta_{f,G}(\xi_\id, v^*(\lfloor\phi\phi'\rfloor\xi_u\cup \xi_{u^{-1}})).
\end{align*}
It thus, suffices to check that 
\[
\eta_\id(\xi_\id, \lfloor v^*(\phi)\rfloor\xi_\id)=\eta_\id(\xi_\id, \lfloor\phi\rfloor\xi_\id).
\]
This equality holds if and only if $v^*(\mathcal{H}_\id)=\mathcal{H}_\id$. The hessian $\mathcal{H}_\id$ is a constant multiple of $x_1^{n-2}\cdot\ldots\cdot x_N^{n-2}$ that is preserved by the permutation of coordinates and diagonal element $g$ acts on it by $(\det g)^{-2}=1$ and the statement follows.

 

\end{proof}

\begin{corollary}\label{corollary: pairing on invariants}
For $G^d\subseteq \SL_N(\CC)$ the pairing $\eta_{f,G}$ restricted to $(\A_{f,G}')^G$ is non--degenerate.
\end{corollary}
\begin{proof}
Let $\alpha\in(\A_{f,G}')^G$ be a nonzero element. 
For any $\beta\in\A'_{f,G}$ we have 
\begin{align*}
 \eta_{f,G}(\alpha, \mathrm{sym}(\beta)) & =\sum_{w\in G} \eta_{f,G}(\alpha, w^*(\beta)) && \textit{(Proposition~\ref{proposition: form G-invariant})}
 \\
 &= \sum_{w\in G} \eta_{f,G}((w^{-1})^*\alpha, \beta) =|G| \cdot \eta_{f,G}(\alpha, \beta). 
\end{align*}
Since $\eta_{f,G}$ is non--degenerate on $\A'_{f,G}$, this equality gives the statement.
\end{proof}

\begin{theorem}
    Let $G^d\subseteq \SL_N(\CC)$. 
    The Hochschild cohomology ring $\ccHH^*(\CC[\bx]\rtimes G,f)$ is a $\ZZ/2\ZZ$--commutative Frobenius algebra.
\end{theorem}
\begin{proof}
    We have by Theorem~\ref{theorem: Shklyarov} that $\ccHH^*(\CC[\bx]\rtimes G,f) \cong (\A'_{f,G})^G$. By Corollary~\ref{corollary: pairing on invariants} the latter algebra is endowed with a non--degenerate pairing. It follows from this corollary and Theorem~\ref{theorem: Frobenius algebra before invariants} that this pairing is also Frobenius w.r.t. $\cup$--product. Now for any $x \in (\A'_{f,G})^G$ and $u \in G$ we have $(u^{-1})^*(x) = x$, what gives us
    \[
        x \cup \lfloor p \rfloor \xi_u = (-1)^{|p||\xi_u|} \lfloor p \rfloor \xi_u \cup (u^{-1})^*(x) = (-1)^{|p||\xi_u|} \lfloor p \rfloor \xi_u \cup x
    \]
    by braided commutativity in $\A'_{f,G}$. This together with Theorem~\ref{theorem: Frobenius algebra before invariants} implies that the bilinear form is also $\ZZ/2\ZZ$--symmetric.

\end{proof}

\section{Examples}

\subsection{Example 1}
Consider $f = x_1^3 + x_2^3 + x_3^3$ with the symmetry groups 
\[
 G_1 := \langle(1,2,3),J\rangle \text{ and } G_2:= \langle(1,2,3),t_1t_2^{-1}\rangle
\]
where we denote $J: =t_1t_2t_3$.
As a vector space we have
\[
    \ccHH^*(\CC[\bx]\rtimes G_k,f) \cong \left(\A'_{f,G_k}\right)^{G_k} = \CC \langle \xi_\id, \lfloor x_1x_2x_3\rfloor\xi_\id, \xi_J, \xi_{J^2} \rangle.
\]

Indeed, for both $G_1$ and $G_2$ the fixed loci of $J$ and $J^2$ are zero, and by Proposition~\ref{action of sym on diag} $(1,2,3)^*\xi_{J^a}=\xi_{J^a}$. We have $(t_1t_2^{-1})^*\xi_{J^a}=\xi_{J^a}$ by Eq~\eqref{diagonal action}.

For any $g \in G_k^d$, the fixed locus of $u = (1,2,3)g$ is one--dimensional giving the sector $\A'_{u} = \CC \langle \lfloor 1 \rfloor \xi_u, \lfloor\widetilde x_u \rfloor\xi_u \rangle$. 
By Corollary~\ref{action of diag on special} we have $J^*(\xi_{(1,2,3)g}) = \zeta_3 \cdot \xi_{(1,2,3)g}$. It follows that $(\A'_u)^J = 0$. The same holds for $u = (1,3,2)g$.

By Eq~\eqref{diagonal action} we have $J^*(\xi_{g})=\zeta_3 \xi_{g}$ giving $(\A'_g)^J = 0$.

In the identity sector we have $\A'_\id = \CC[x_1, x_2, x_3]/(x_1^2, x_2^2, x_3^2)\cdot\xi_\id$ with the only $J$-invariant elements being $\xi_\id$ and $\lfloor x_1x_2x_3\rfloor \xi_\id$, which are also $G_k$--invariant.

The only non--vanishing product of $\ccHH^*(f,G_k)$ not involving $\xi_\id$ is
\[
    \xi_J\cup\xi_{J^2}=\left(\frac{3}{\zeta_3-1}\right)^3\lfloor x_1x_2x_3\rfloor\xi_\id 
\]

In mirror symmetry the algebra $\ccHH^*(\CC[\bx]\rtimes G_k,f)$ is considered as a B--side model. Its mirror A--side model is given by a smooth elliptic curve 
\[
 C_f =\lbrace(x_1:x_2:x_3) \in \PP^2 |x_1^3+x_2^3+x_3^3=0\rbrace.
\]
Let $G'_k$ be the image of $G_k$ in $\mathrm{PGL}_3$. The nontrivial elements of $G'_k$ act on $C_f$ without fixed points, giving a trivial action on the cohomology ring $\mathrm{H}^*(C_f,\CC)$. It follows that the Chen-Ruan cohomology ring of $C_f$ is isomorphic to the ordinary cohomology ring for both groups. We conclude that there is a mirror symmetry isomorphism $\ccHH^*(\CC[\bx]\rtimes G_k,f)\simeq \mathrm{H}^*(C_f,\CC)$.

\subsection{Example 2}
Consider $f = x_1^4 + x_2^4 + x_3^4+x_4^4$ with the symmetry group
\[
 G := \langle J, (1,2)(3,4), (1,3)(2,4)\rangle
\]
where we denote $J = t_1t_2t_3t_4$.

The fixed loci of elements of the form $(\sigma, J^a)$ are two--dimensional if $\sig\ne\id$, $a\equiv 0\mod 2$ and are equal to zero for all other nonidentity elements.

For $a=1,2,3$, the vectors $\xi_{J^a}$ are fixed by the action of $G$ by Eq~\eqref{diagonal action} and Proposition~\ref{action of sym on diag}. These vectors span the $3$--dimenional subspace of $\ccHH^*(\CC[\bx]\rtimes G,f)$.

For $a = 1,3$ we have by  Proposition~\ref{trans on trans} and Proposition~\ref{action of sym on diag} $$((1,2)(3,4))^*(\xi_{((1,2)(3,4), J^a)})=\xi_{((1,2)(3,4), J^a)}, \enskip ((1,3)(2,4))^*(\xi_{((1,2)(3,4), J^a)})=\xi_{((1,2)(3,4), J^a)}$$ and by Corollary~\ref{action of diag on nonspecial} also 
$$
    J^*(\xi_{((1,2)(3,4), J^a)})=\xi_{((1,2)(3,4), J^a)}.
$$ 
We conclude that the vectors $\xi_{(\sig, J^a)}$ with $\sig\ne\id$ and $a = 1,3$ are $G$-invariant. There are total $6$ of those.

For $a = 0, 2$ by Corollary~\ref{action of diag on special} we have 
$$
J^*(\xi_{((1,2)(3,4), J^a)})=-\xi_{((1,2)(3,4), J^a)}.
$$ 
Furthermore by  Proposition~\ref{trans on trans} we have 
\begin{align*}
& ((1,2)(3,4))^*(\xi_{((1,2)(3,4), J^a)})=\xi_{((1,2)(3,4), J^a)}, 
\\
& ((1,3)(2,4))^*(\xi_{((1,2)(3,4), J^a)})=-\xi_{((1,2)(3,4), J^a)}.
\end{align*}
It follows that the only fixed vector in $((1,2)(3,4), J^a)$--th sector is 
\[
    \lfloor (x_1+x_2)^2-(x_3+x_4)^2\rfloor \xi_{((1,2)(3,4), J^2)}. 
\]
The same holds for all the sectors of the form $(\sigma, J^a)$ with $\sig\ne\id, a=0,2$. There are total $6$ of such sectors spanning in total the $6$--dimensional $G$--invariant vector space.

It remains to compute the $G$-invariants in the identity sector. Elements $\lfloor (x_1x_2x_3x_4)^a\rfloor\xi_\id$ for $a=0,1,2$ are clearly invariant. The $J$-invariant vectors are given by the homogenous polynomials of degree $4$. Since $\lfloor x_i^3 \rfloor =0$, the $G$--invariants are obtained by symmetrizations of either 
$x_i^2x_jx_k$ or $x_i^2x_j^2$. We get $3$ total invariant elements of each type:
\begin{align*}
    & \lfloor x_2 x_3 x_1^2+x_2^2 x_4 x_1+x_3^2 x_4 x_1+x_2 x_3 x_4^2 \rfloor \xi_\id,
    \
    \lfloor x_2 x_4 x_1^2+x_3 x_4^2 x_1+x_2^2 x_3 x_1+x_2 x_3^2 x_4 \rfloor \xi_\id,
    \\
    &\quad \quad \lfloor x_3 x_4 x_1^2+x_2 x_3^2 x_1+x_2 x_4^2 x_1+x_2^2 x_3 x_4 \rfloor \xi_\id;
    \\
    & \lfloor x_1^2 x_2^2+x_3^2 x_4^2 \rfloor \xi_\id, \quad \lfloor x_1^2 x_3^2+x_2^2 x_4^2 \rfloor \xi_\id, \quad \lfloor x_2^2 x_3^2+x_1^2 x_4^2 \rfloor \xi_\id.
\end{align*}

We see that $\ccHH^*(\CC[\bx]\rtimes G)$ is $24$-dimensional in this case. The only nonzero products by the vectors not in the identity sector are 
\begin{align*}
& \xi_{J}\cup\xi_{J^{-1}}=\left(\frac{4}{\zeta_4-1}\right)^4\lfloor(x_1x_2x_3x_4)^2\rfloor\xi_\id,
\\
&\lfloor (x_1+x_2)^2-(x_3+x_4)^2\rfloor\xi_{(1,2)(3,4)J^a}\cup\lfloor (x_1+x_2)^2-(x_3+x_4)^2\rfloor\xi_{(1,2)(3,4)J^{-a}}
\\
& \quad =\lfloor (x_1^2-x_3^2)^2\Phi_{12}\Phi_{34}\rfloor\xi_\id=-32\lfloor(x_1x_2x_3x_4)^2\rfloor\xi_\id,
\\
& \xi_{((1,2)(3,4), J^a)}\cup \xi_{((1,2)(3,4), J^a)}=16 \left(\frac{4}{\zeta_4^{2a}-1}\right)^2\lfloor(x_1x_2x_3x_4)^2\rfloor\xi_\id=64\lfloor(x_1x_2x_3x_4)^2\rfloor\xi_\id
\end{align*}
where $a \equiv 0 \mod 2$.

\subsection{Example 3}
Consider $f = x_1^4 + x_2^3 + x_3^3$ with the groups $G^d = \langle g \rangle$ and $G^s := \langle \sigma\rangle$ for $g = t_2^2t_3$, $\sigma := (2,3)$. In \cite{BTW17} it was shown that there is a $\CC$--algebra isomorphism $\tau_0: \Jac(f) \cong \ccHH^*(\CC[\bx]\rtimes G^d,f)$. We show that $\ccHH^*(\CC[\bx]\rtimes G^s,f) \cong \ccHH^*(\CC[\bx]\rtimes (G^s \ltimes G^d),f)$.
\\
\\
The algebra $\ccHH^*(f, G^d)$ is generated by $\lfloor x_1 \rfloor \xi_\id$, $\xi_g$, $\xi_{g^2}$ subject to:
\[
    (\lfloor x_1 \rfloor \xi_\id)^{3} = 0, \ (\xi_g)^{2} = 0, \ (\xi_{g^2})^{2} = 0
\]
ang the isomorphism $\tau_0$ is just given by
\[
    \lfloor x_1 \rfloor \mapsto \frac{1}{3} \lfloor x_1 \rfloor \xi_{\id}, 
    \
    \lfloor x_2 \rfloor \mapsto \frac{1}{\sqrt{3}} \xi_{g}, 
    \
    \lfloor x_3 \rfloor \mapsto \frac{1}{\sqrt{3}} \xi_{g^2}.
\]
Under this map we have also
\begin{align}
    \lfloor x_2 + x_3\rfloor \mapsto \frac{1}{\sqrt{3}} (\xi_g + \xi_{g^2}), \
    \lfloor x_2x_3\rfloor \mapsto \frac{1}{3} \xi_g \cup \xi_{g^2} = 3 \lfloor x_2x_3 \rfloor \xi_\id.
\end{align}

Denote $\A' := \ccHH^*(\CC[\bx]\rtimes G^s,f)$ and $\A'' := \ccHH^*(\CC[\bx]\rtimes (G^s \ltimes G^d),f)$.
The algebra $\A'$ is generated by
\[
    v'_{\alpha,0} := \lfloor x_1^\alpha \rfloor \xi_\id, \ v'_{\alpha,1} := \lfloor x_1^\alpha (x_2+x_3) \rfloor \xi_\id, \ v'_{\alpha,2} := \lfloor x_1^\alpha (x_2x_3) \rfloor \xi_\id, \quad 0 \le \alpha \le 2.
\]
with a 'usual' $\Jac$--product. In particular $(v'_{1,0})^3 = 0$, $(v'_{0,1})^2 = 2 v'_{0,2}$ and $(v'_{0,2})^2 = 0$.
\\
\\
Concerning the algebra $\A''$ note that 
\begin{align*}
\sigma^*(\xi_\sigma) &= - \xi_\sigma, \quad  (\sigma g)^* (\xi_{\sigma  g}) = - \xi_{\sigma  g}, \quad (\sigma g^2)^* (\xi_{\sigma  g^2}) = - \xi_{\sigma  g^2},
\\
\sigma^*(\xi_g) &= \xi_{g^2}, \quad \sigma^*(\xi_{g^2}) = \xi_{g}.
\end{align*}
Therefore the algebra $\A''$ is generated by 
\[
    v''_{\id; \alpha, 0} := \lfloor x_1^\alpha \rfloor \xi_\id, 
    \ 
    v''_{\id; \alpha, 2} := \lfloor x_1^\alpha \cdot x_2x_3 \rfloor \xi_\id, \quad 0 \le \alpha \le 2,
\]
and
\[
    v''_{tw,\alpha} := \lfloor x_1^\alpha \rfloor (\xi_{g} + \xi_{g^2}). \quad 0\le\alpha\le 2
\]
We have $H_{g,g^{-1}} = 9 \lfloor x_2x_3 \rfloor$ and 
\[
    v''_{tw,\alpha} \cup v''_{tw,\beta} = - 6  \cdot v''_{\id,\alpha+\beta,2} \cdot \delta_{\alpha+\beta \le 2}.
\]
Consider the map $\A' \to \A''$ given by
\begin{align}
    v'_{\alpha,0} \mapsto v''_{\id; \alpha, 0}, 
    \ v'_{\alpha, 1} \mapsto \frac{\sqrt{-1}}{\sqrt{3}} \cdot v''_{tw; \alpha}, 
    \ v'_{\alpha, 2} \mapsto v''_{\id; \alpha, 2}.
\end{align}
The computations above show that this is an isomorphism. 
Note that it is essentially the restriction of the isomorphism $\tau_0$ to the $G^s$--fixed subspace.

\end{document}